\documentclass[oneside,reqno]{amsart}

\usepackage{Macros}
\usepackage{physics}
\usepackage{mathtools}
\usepackage{appendix}
\usepackage{bm}
\usepackage{bbm}
\usepackage{mathrsfs}

\usepackage{hyperref}

\usepackage{parskip}

\newtheorem{proposition}{Proposition}
\newtheorem{theorem}[proposition]{Theorem}
\newtheorem{lemma}[proposition]{Lemma}

\theoremstyle{remark}

\theoremstyle{definition}
\newtheorem{definition}[proposition]{Definition}

\numberwithin{equation}{section}
\numberwithin{proposition}{section}
\numberwithin{table}{section}

\renewcommand{\epsilon}{\varepsilon}
\DeclareMathOperator*{\argmin}{arg \, min}

\DeclareMathOperator{\diff}{d}
\DeclareMathOperator{\Lip}{Lip}
\DeclareMathOperator{\loc}{loc}
\DeclareMathOperator{\AC}{AC}
\DeclareMathOperator{\Prox}{Prox}
\DeclareMathOperator{\id}{id}

\newcommand{\scr}[1]{\mathscr{#1}}

\title{Li-Yau-Hamilton Inequality on the JKO Scheme for the Granular-Medium Equation}
\author{Fanch Coudreuse}
\address{Institut Camille Jordan, Lyon 1}

\begin{document}

\begin{abstract}
    We establish a version of the Li--Yau--Hamilton inequality for the Granular-Medium equation on the torus, both at the PDE level and for its time-discrete approximation given by the JKO scheme. We then apply this estimate to derive further quantitative results for the continuous and discrete JKO flows, including Lipschitz and $L^\infty$ bounds, as well as a quantitative Harnack inequality. Finally, we use the regularity provided by this estimate to show that the JKO scheme for the Fokker--Planck equation converges in $L^2_{\mathrm{loc}}((0,+\infty); H^2(\bb{T}^d))$.
\end{abstract}

\maketitle

\tableofcontents

\section{Introduction}

     The goal of this article is to derive Li–Yau–Hamilton type inequalities for the JKO approximation to the Granular–Medium equation
\[
    \partial_t \rho_t = \Delta \rho_t + \nabla \cdot (\rho_t \nabla V + \rho_t \nabla W * \rho_t)
\]
with $C^{2,1}(\bb{T}^d)$ potentials, on the torus $\bb{T}^d$. At the same time, we derive a new, so far as we know, Li–Yau–Hamilton inequality also at the level of the continuous-time solution. As a by-product, we show how to use this inequality to derive quantitative estimates for the solution, and we shall prove a quantitative version of the Harnack inequality, both at the continuous and JKO levels, and show that one can use these estimates to prove strong, local-in-time convergence of the JKO scheme.

Li–Yau type inequalities are fundamental estimates in the theory of the heat equation, and more generally in the study of diffusion equations. Originally proved by P. Li and S. Yau in \cite{Li-Yau}, they state that positive solutions to the heat equation on Riemannian manifolds $M^d$ with non-negative Ricci curvature satisfy the pointwise bound $\Delta \log \rho_t \geq -\frac{d}{2 t}$. Later on, this inequality was improved by Hamilton, under more stringent assumptions on the geometry of $M^d$ to a full Hessian estimate by Hamilton in \cite{Hamilton_Matrix_Ineq}, taking the form $D^2 \log \rho_t \succeq -\frac{1}{2t}$. This powerful estimate allows one to derive quantitative versions of the Harnack inequality, as well as other quantitative and qualitative properties of the heat equation. Since then, generalizations of this estimate to other types of diffusions have attracted great attention from various communities, including Markov diffusion operators using logarithmic type inequalities, first by D. Bakry and M. Ledoux \cite{BakLed}, then extended by D. Bakry, F. Bolley and I. Gentil in \cite{LiYauCD}, and PDEs. Let us mention for instance the Aronson-Bénilan estimate for the porous-medium equation \cite{AronsonBenilan}. 

It is nowadays well understood, since the fundamental work of Jordan, Kinderlehrer, and Otto \cite{JKO}, that the heat equation is the gradient flow of the Boltzmann entropy in Wasserstein space, and that one can use a time-discrete implicit-Euler-type scheme to approximate solutions of this equation. This scheme, popularised under the name of the JKO scheme, can be applied to various functionals on the space of probability measures, and is based on the iterative minimisation
\[
    \rho_{k+1}^\tau \in \argmin \cl{F} + \frac{1}{2 \tau} W_2^2(\cdot,\rho_k^\tau)
\]
Without stringent assumptions on $\cl{F}$ the convergence of this scheme is typically weak. A popular strategy to improve convergence, and to show the robustness of the scheme, is to prove that well-known estimates which hold for the continuous equation also hold at the level of the JKO scheme. See for instance \cite{Lee_Lipschitz} \cite{Lip_JKO_Filippo} \cite{Caillet_Continuity} \cite{L2H2} \cite{LpSobo}. 

The Li–Yau–Hamilton inequality is then a natural choice of such an estimate one would like to obtain at the level of the JKO scheme. A first step in this direction was obtained by P. W. Y. Lee in 2018 in \cite{Lee_Harnack}: on the torus, starting from $C^{2,\alpha}(\bb{T}^d)$ and strictly positive initial data, he showed that an estimate of the form $D^2 \log \rho_t^\tau \succeq -\frac{C}{2 t}$ for all $\tau \leq \tau_0$, with $\tau_0$ depending on the initial data, $C \in (1/2,1]$ is some universal constant and $(\rho_t^\tau)_{t \geq 0}$ is the piecewise-constant interpolation, with step $\tau$ of the values obtained from the JKO scheme.

This is the first hint that the full Li-Yau-Hamilton estimate might be recovered for the JKO scheme. At least three directions of improvement can be listed: Firstly, in the continuous case, the estimate holds independently on the regularity of the initial data, in contrast with Lee's result where $\tau_0$ blows up as $\rho_0$ becomes less and less regular. Secondly, one would hope to recover the optimal constant $1/2$, at least asymptotically, in the sense that one may hope to obtain a constant $C_\tau$ going to $1/2$ as $\tau \to 0$. The final possible improvement is when the initial data is already regular, in which case the Li-Yau-Hamilton can be slightly improved to an inequality valid up to time $t=0$. 

Another direction is to try to extend the proof to other types of equations, for instance the Fokker–Planck equation. To obtain such a result, one should at least be able to prove it for the continuous-time equation. Unfortunately, if one mimics the classical maximum-principle argument for $D^2(\log \rho_t + V)$ (as $\log \rho_t + V$ is the natural pressure associated to the equation), then one ends up with gradient terms that are hard to control. This can be resolved by looking instead at the quantity $D^2(\log \rho_t + \frac{1}{2} V)$ but it demands that $V$ is convex, and that a fourth-order quantity involving $V$ is bounded from below, we refer to \cite{LeeDiffHarnack} for more about such estimates, see also \cite{ChaConfEich} for a deeper study on this direction. On the Torus unfortunately, such assumptions cannot hold, as there is no non-constant convex function. On the other hand, the fundamental, but elementary, observation, in this case, gradients terms can be controlled using the semi-convexity:

\begin{lemma}[Gradient Estimate for Semi-Convex Periodic Function] \label{lemma: gradient_estimate}
    Let $u : \bb{R}^d \to \bb{R}$ be such that $D^2 u \succeq -\lambda \rm{I}_d$ in the weak-sense, with $\lambda \geq 0$ (i.e. $u$ is $-\lambda$-convex). Then one has $||\nabla u(x)||_\infty \leq \frac{1}{2} \lambda$ for any $i=1,\ldots,d$ (where $\nabla u(x)$ is understood as any element of $\partial u(x)$, and $||v||_\infty = \max_i |v_i|$ is the max norm of a vector $v$). 
\end{lemma}
    
\begin{proof}
    Working with $-u$ instead of $u$ we only need to consider the semi-convex case. Let $\nabla u(x) \in \partial u(x)$, then by semi-convexity one has for any $y \in \bb{R}^d$
    \begin{align}
        u(y) \geq u(x) + \nabla u(x) \cdot (y-x) - \frac{\lambda}{2} |x-y|^2 
    \end{align}
    applying this inequality to $x+\pm e_i$ for $i=1,\ldots,d$, and $e_i$ the $i$-th element of the canonical basis yields
    \begin{align}
        \pm \partial_i u(x) \leq u(x\pm e_i) - u(x) + \frac{\lambda}{2} = \frac{\lambda}{2}
    \end{align}
    where we used periodicity of $u$, which concludes the proof. 
\end{proof}

This lemma is going to be the main tool to extend the Li-Yau-Hamilton estimate for the Fokker-Planck, but also for the Granular-Medium equation, on the torus, both at the continuous and JKO level. 

\subsection{Main results}
    
    Before presenting the main result, we shall introduce our assumptions. In this paper, we shall consider two potentials $V,W$ which will always be at least of class $C^{2,1}(\bb{T}^d)$ (and sometimes more regular). We shall quantify this $C^{2,1}$ regularity $V,W$ using the following constants:
    \begin{itemize}
        \item \emph{Semi-Convexity bound: } We have $D^2 V \succeq -\lambda_V$ and $D^2 W \succeq -\lambda_W$ for some $\lambda_V,\lambda_W \geq 0$. We shall also set $\lambda^* := \lambda_V + \lambda_W$. 
        \item \emph{Lipschitz bound: } For all $\nu \in \bb{S}^d$, we have $|\nabla \partial_{\nu \nu} V|_1 \leq L_V$ and $|\nabla  \partial_{\nu \nu} W|_1 \leq L_W$ with $L_V,L_W \geq 0$. We shall also set $L^* := \frac{1}{2} L_V + L_W$. 
    \end{itemize}
    Finally, we shall make use of the following constant, combining the semi-convexity and Lipschitz behaviour:
    \begin{equation} \label{eq: Lambda_Star}
       \Lambda = 2 \lambda^* + L^*
    \end{equation}

    Our main result is an asymptotic version of the Li-Yau-Hamilton inequality for the Granular-Medium equation on the torus. Under $C^{2,1}$ assumptions on the potentials $V$ and $W$. Informally, it takes the following form:

    \begin{theorem}[Asymptotic Li-Yau-Hamilton estimate]
        Let $(\rho_t^\tau)_{t \geq 0}$ be a JKO flow starting from $\rho_0$. Then for all $t_0 > 0$, and $t \geq t_0$ one has
        \[
            D^2 (\log \rho_t^\tau + V + W * \rho_t^\tau) \geq\left \{ \begin{array}{ll} -\frac{1+o(\tau)}{2 t} & \mbox{if $\Lambda = 0$} \\ -\frac{(1+o(\tau))\Lambda}{2 (1-e^{-\Lambda t})} & \mbox{else} \end{array}  \right . 
        \]
    \end{theorem}

    See Section 3 theorem \ref{thm: Asymptotic_Li_Yau_Hamilton_Estimate} for the precise statement. In fact we show that, under regularity assumptions on the initial data, this estimate can be improved up to time $0$. This is an improvement of Lee's result on the four directions we explained before. As a by-product of this result, letting $\tau \to 0$, we obtain a version of the Li-Yau-Hamilton estimate for the Granular-Medium equation on the torus. We shall nevertheless provide a direct proof of this estimate on the continuous level for two reasons: First we believe that it might be of interest to use the same method for similar equations that does not admits a Gradient flow structure. Second the estimate is easier to obtain in the continuous setting, and the proof gives a hint on the computations one want to mimic at the discrete level. We shall note however that it was easier to see how to handle the addition of the interaction term at the discrete level, and it was only after that we found a proof at the continuous level, in contrast with what happen usually when dealing with the JKO scheme. 

    \begin{theorem}[Li-Yau-Hamilton estimate]
        Let $\rho$ be a Gradient flow solution of the Granular-Medium equation (see \ref{thm: gradient_flow_solution} for the relevant definition). Then for all $t > 0$ one has
        \[
            D^2 (\log \rho_t + V + W * \rho_t) \geq -\left \{ \begin{array}{ll} \frac{1}{2 t} & \mbox{if $\Lambda = 0$} \\ \frac{\Lambda}{2 (1-e^{-\Lambda t})} & \mbox{else} \end{array}  \right .
        \]
    \end{theorem}

    To our knowledge, this version of the Li-Yau-Hamilton inequality is new. 
    
    As for the classical heat equation, it can be used to derive quantitative estimates for solutions of the Granular-Medium equation. We shall present three of them: a Lipschitz and $L^\infty$ bound, and a quantitative Harnack inequality.

    Similarly, the discrete version of the estimate can be used to derive Lipschitz, $L^\infty$ and Harnack estimates for the discrete flow, uniform in $\tau$. We shall then use these results to improve the convergence of the JKO scheme, locally in time (the first convergence being an almost immediate consequence of the uniform Lipschitz estimate). 

    \begin{theorem}
        Let $\rho_0$ be such that $\cl{F}[\rho_0] < +\infty$. Then:
        \begin{enumerate}
            \item For all $1 \leq p < +\infty$ and $\alpha \in (0,1)$, $(\rho^\tau_t)_{t \geq 0}$ converges to a solution to the Granular-Medium equation starting from $\rho_0$ in $L^p_{\loc}((0,+\infty);C^{0,\alpha}(\bb{T}^d))$.
            \item If we also assume that $W = 0$ (i.e. we work with the Fokker-Planck equation). Then one has also convergence in $L^2_{\loc}((0,+\infty);H^2(\bb{T}^d))$. 
        \end{enumerate}
    \end{theorem}
    
\subsection{Organization of the paper}

    As explained in the introduction, we chose to still keep an (almost) self-contained proof of the estimate in the continuous time case. As such, the paper is divided into two parts: Section $2$ deals with the continuous time case, and Sections $4,5$ with the discrete time case. Some of the proofs are postponed to the appendix, as they are merely technical and do not involve particularly nice ideas. The precise organization is the following one:

    \begin{itemize}
        \item In Section $2$ we prove the Li-Yau-Hamilton estimate for the Granular-Medium equation. We then proceed to use this inequality to prove quantitative Lipschitz and $L^\infty$ estimates for the solution. Finally, mimicking the classical proof, we show a quantitative Harnack inequality for solutions to the Granular-Medium equation. 
        
        \item In Section $3$ we collect the relevant basic tools from the theory of optimal transport and basics results about the JKO scheme. 
        
        \item In Section $4$ we prove the asymptotic Li-Yau-Hamilton estimate for the JKO scheme. A big part of the proof is merely technical and is postponed to Appendix $A$ and $B$. 
        
        \item In Section $5$ we use the estimate to derive estimates on the discrete case, analogue to the continuous time estimate: Lipschitz and $L^\infty$ bounds, together with a quantitative Harnack inequality. We then use these estimates to derive the local in time strong convergence of the scheme. 
        
        \item In Appendix $A$ we show how to prove rigorously the one-step-improvement \ref{thm: One-Step_Improvement} of semi-convexity used in Section $4$ when the initial data is irregular.
        
        \item In Appendix $B$ we prove the asymptotic estimates used in Section $4$. 
    \end{itemize}
    
\subsection{Notations and Conventions}
    In the rest of the paper, we shall adopt the following notations and conventions: 
    
    \begin{itemize}
        \item For $x \in \bb{R}^d$, we let $[x]$ be the unique representative of $x$ mod $\bb{Z}^d$ on the cube $Q = [-1/2,1/2)^d$. We shall denote by $d(x,y)^2$ the metric of the torus, such that if $x,y \in \bb{R}^d$, one can write $d(x,y)^2 = |[x-y]|^2$. Note that whenever $|x-y|_\infty < 1/2$, then $d(x,y)^2 = |x-y|^2$. 
    
        \item We shall not make distinctions between class of functions defined on the torus, and the corresponding class of periodic functions on $\bb{R}^d$. For instance, a $C^{k,\alpha}(\bb{T}^d)$ function is the same as a periodic $C^{k,\alpha}(\bb{R}^d)$ function. Similarly, we shall not make distinctions between a probability measure on the torus is the same as a translation invariant positive measure on $\bb{R}^d$ giving mass one to the unit cube.

        \item We shall abuse notations by denoting by $\rho$ a measure, and its density with respect to the Lebesgue measure if it has one. 
        
        \item If $v \in \bb{R}^d$, we let $|v|_1,|v|_2,|v|_\infty$ be respectively the $l^1,l^2$ and $l^\infty$ norms of $v$. If $U : \bb{T}^d \to \bb{R}^d$ is some function, we shall write $|U|_k$ for $\sup_{x \in \bb{T}^d} |U(x)|_k$ where $k = 1,2,\infty$.
        
        \item Spatial derivatives in some direction $\nu \in \bb{S}^d$ will always be denoted by a subscript. On the other hand, we shall reserve the time subscript for the value of some function at time $t$, to emphasis that we interpret solutions as paths valued in some function space. For instance $\rho_{t,\nu}$ is the derivative of $\rho$ in direction $\nu$ evaluated at time $t$, but $\rho_t$ is the value of $\rho$ at time $t$ (seen as some function of x), and $\partial_t \rho_t$ is the time derivative evaluated at $t$.
        
        \item If $A$ is a symmetric matrix, and $\alpha$ a scalar, we write $A \succeq \alpha$ to mean $A \succeq \alpha \rm{I}_d$. Furthermore, if $v \in \bb{R}^d$, we shall write $A[v,v]$ for the quantity $v^T A v$. 
    \end{itemize}

\subsection{Acknowledgment} The author acknowledges the support of the European Union via the ERC AdG
101054420 EYAWKAJKOS. \\
The author would also like to thank Filippo Santambrogio and Ivan Gentil for valuable discussions and feedbacks during this work, as
well as Louis-Pierre Chaintron for pointing out the existence of Hamilton inequality for Fokker-Planck equation in the whole space and Aymeric Baradat for suggesting to look at the asymptotic equivalent of $(E_k)_{k \geq 0}$ in the heat case.

\section{Li-Yau-Hamilton Estimate for The Granular-Medium}

    The Granular-media equation on the torus is the equation
\begin{align} \label{eq: Aggregation-Diffusion} 
    \left \{ \begin{array}{ll} 
        \partial_t \rho_t = \Delta \rho_t + \nabla \cdot (\rho_t \nabla V + \rho_t \nabla W * \rho_t) & \mbox{on $(0,+\infty) \times \bb{T}^d$} \\
        \rho_{t=0} = \rho_0 & \mbox{}
    \end{array} \right .
\end{align}
We shall see $V$ as a potential energy, and $W$ as a potential of interaction. This interpretation comes from the McKean-Vlasov SDE 
\begin{equation}
    \dd{X_t} = -\nabla V(X_t) \dd{t} - \nabla W * \cl{L}[X_t](X_t) \dd{t} + \sqrt{2} \dd{W_t}
\end{equation}
where $\cl{L}[X_t]$ is the law of $X_t$, and $W$ is a standard Brownian motion on the torus. This equation appears in the mean-field regime for a weakly interacting cloud of particles. It is not hard to see using Itö's formula that if $X$ solves the above equation, then the law of $X$ weakly solves the Granular-Media equation.

This family of equations encompasses at least two famous equations:
\begin{itemize}
    \item The Heat equation when $V=W=0$, $\partial_t \rho_t = \Delta \rho_t$.
    \item The Fokker-Planck equation $W=0$, $\partial_t \rho_t = \Delta \rho_t + \nabla \cdot \rho_t \nabla V$.
\end{itemize}
In the following, when we want to restrict to one of those cases, we shall write down "the Heat case" or "the Fokker-Planck case".

This equation is in fact part of the broader family of Aggregation-Diffusion equation, where we replace the Laplacian by the non-linear diffusion term $\Delta \Psi[\rho_t]$ for some function $\Psi$. A popular choice is $\Psi[t] = t^m$ which lead to equation of porous media or fast-diffusion type. For more information on the Aggregation-Diffusion equation, we refer to the introduction to the topic by Gómez-Castro \cite{GuideAggregDiffusion}. For physical derivation of the equation from particle systems one can consult the extensive survey by Chaintron and Diez \cite{PropagChaosI} \cite{PropagChaosII}. 

\subsection{Notion of Solution}

    We shall be concerned with solution arising as gradient flow of the energy 
    \begin{equation}
        \cl{F}[\rho] = \int_{\bb{T}^d} [\log \rho + V + W * \rho] \diff \rho
    \end{equation}
    with respect to the Wasserstein metric on the torus in the sense of Ambrosio, Gigli and Savaré \cite{AGS}. By an adaptation of the corresponding theorem 11.2.8. in the above book (\cite{AGS}) we have the following existence result.

    \begin{theorem}[Gradient-Flow solutions \cite{AGS}] \label{thm: gradient_flow_solution}
        For any $\rho_0 \in \cl{P}(\bb{T}^d)$, there exists a unique gradient flow solution for $\cl{F}$ starting from $\rho_0$. That is a curve
        \[ 
            \rho_{\cdot} \in C([0,+\infty);\bb{W}_2) \cap C^{0,1}_{\loc}((0,+\infty);\bb{W}_2) \cap \AC_2((t_0,T);\bb{W}_2)
        \]
        with $\rho_{|t=0} = \rho_0$ such that
        \begin{enumerate}
            \item For any $t > 0$, $\rho_t \ll \cl{L}^d$ and $\rho_t > 0$ (as a density).
            \item $\nabla \rho_{\cdot} \in L^1_{\loc}((0,+\infty);W^{1,1}(\bb{T}^d))$. 
            \item The Fisher's information is locally square integrable, i.e.
            \[
                t \to \int_{\bb{T}^d} \left | \frac{\nabla \rho_t}{\rho_t} + \nabla V + \nabla W * \rho_t \right |^2 \diff \rho_t \in L^2_{\loc}((0,+\infty))
            \]
            \item $\rho_{\cdot}$ is a weak distributional solution of the granular-media equation, i.e. for all $\psi \in C^2(\bb{T}^d)$ one has
            \[
                \frac{\diff}{\diff t} \int_{\bb{T}^d} \psi \diff \rho_t = \int_{\bb{T}^d} [\Delta \psi - \nabla V \cdot \nabla \psi - \nabla W * \rho_t \cdot \nabla \psi] \diff \rho_t
            \]
        \end{enumerate}
    \end{theorem}

    \begin{remark}
        If $\rho_0$ is of finite energy (i.e. $\cl{F}[\rho_0] < +\infty$), the existence result is a by-product of the convergence of the JKO scheme (theorem \ref{thm: JKO_converges}), and the curve is actually of class $C^{0,1/2}_{\loc}([0,+\infty);\bb{W}_2)$. 
    \end{remark}

    We also have the following regularity for solutions. 

    \begin{proposition}[Regularity of gradient flow solution]
        Let $\rho_{\cdot}$ be any gradient flow solution to the Granular media equation, then for any $\alpha < 1$, $\rho_{\cdot}$ belongs to the class $C^{1,\alpha/2}_{\loc}((0,+\infty);C^{2,\alpha}(\bb{T}^d))$. \\
        Furthermore, if $\rho_0,V,W$ are smooths, then $\rho_{\cdot}$ is smooth (up to $t=0$). 
    \end{proposition}

    \begin{proof}
        By the parabolic theory for measures solutions to Fokker-Planck equation with bounded coefficients \cite{FokkerPlanckRegu}, $\rho_{\cdot}$ is of class $L^\infty_{\loc}((0,+\infty);L^\infty(\bb{T}^d))$. Furthermore, using that $t \to \rho_t$ is locally Lipschitz when seen as a curve valued in $\bb{W}_2$, one can easily see that $\Phi(t,x) := \nabla V(x) + \nabla W * \rho_t(x)$ is actually of class $C^{0,1}_{\loc}((0,+\infty);C^{1,1}(\bb{T}^d))$. Using Schauder theory for parabolic equation with bounded solution and Hölder coefficients we obtain that $\rho_{\cdot}$ is of class $C^{1,\alpha/2}_{\loc}((0,+\infty);C^{2,\alpha}(\bb{T}^d))$ for any $\alpha < 1$. 

        If $\rho_0,V,W$ are smooth, one can extends the above regularity up to time $t=0$, then using a bootstrap argument we obtain the global regularity. 
    \end{proof}

\subsection{Li-Yau-Hamilton Inequality for Granular-media equation}

    We state here second (in the introduction) main result: the Li-Yau-Hamilton inequality for solution to the Granular-media equation. We shall first give a proof in the smooth case, as the method might be useful in general context. In the un-regular case, this will be obtained as a by-product of the convergence of the JKO scheme using the asymptotic version of the estimate \ref{thm: Asymptotic_Li_Yau_Hamilton_Estimate}.   

    \begin{definition}[Pressure function]
        Let $\rho \in \cl{P}(\bb{T}^d)$ such that $\rho$ admits a strictly positive density. We define the associated pressure variable by the formula
        \begin{equation}
            u[\rho] := \log \rho + V + W * \rho 
        \end{equation}
        
        For a general measure $\rho \in \cl{P}(\bb{T}^d)$, we shall write down $D^2 u[\rho] \succeq -\lambda_0$ with the following meaning: if $\lambda_0 < +\infty$, this means that $u[\rho]$ is well-defined and the inequality is understood in the semi-convex sense, and by abuse of notation, if $\lambda_0 = +\infty$ this shall not bear any meaning (i.e. $\rho$ can be any probability measure even if $u[\rho]$ is not well defined).
    \end{definition}

    \begin{remark}
        If $D^2 u[\rho] \succeq -\lambda_0$ with $\lambda_0 < +\infty$, then $\rho$ is $C^{0,1}(\bb{T}^d)$, as $\rho$ is the exponential of a Lipschitz function. 
    \end{remark}
    
    Observe that if $D^2 u[\rho] \succeq -\lambda_0$ with $\lambda_0 < +\infty$, then has a positive $C^{0,1}(\bb{T}^d)$ density (as it is continuous, with Lipschitz logarithm). If $\rho_\cdot$ is a gradient-flow solution to the Granular-media equation, then the pressure $u[\rho_t]$ always defines a proper function, as $\rho_t \ll \cl{L}^d$ and $\rho_t > 0$ a.e. The following lemma follows from algebraic computations. 
    
    \begin{lemma} \label{lemma: HJ_Pressure_Variable}
        If $\rho_\cdot$ is a solution smooth solution to the Granular-media equation, and $\rho_0 > 0$, then $u_t = u[\rho_t]$ satisfies the non-local Hamilton-Jacobi equation 
        \begin{equation} \label{eq: HJ_Pressure_Variable}
            \partial_t u_t = \Delta u_t + |\nabla u_t|^2 - \nabla V \cdot \nabla u_t - \cl{R}(\nabla W, \nabla u_t, \rho_t) \qquad u_{|t=0} = u[\rho_0]
        \end{equation}
        where $\cl{R}(\nabla W, \nabla u_t, \rho_t)$ is a non-local term given by
        \begin{equation}
            \cl{R}(\nabla W, \nabla u_t, \rho_t)(x) := \int_{\bb{T}^d} \nabla W(x-y) \cdot [\nabla u_t(x) - \nabla u_t(y)]] \diff \rho_t(y)
        \end{equation}
    \end{lemma}

    The Li-Yau-Hamilton inequality, then takes the following form.
    
    \begin{theorem}[Li-Yau-Hamilton Inequality for Granular-media Equation]
        Let $\rho$ be any gradient flow solution to the Granular-media equation starting from $\rho_0 \in \cl{P}(\bb{T}^d)$. Consider $\Lambda$ as defined in equation \ref{eq: Lambda_Star}.  Suppose that $D^2 u[\rho_0] \succeq -\lambda_0$ with $\lambda_0 \in [0,+\infty]$. Then
        \begin{enumerate}
            \item If $\lambda_0 < +\infty$, one has
            \begin{equation}
                D^2 u[\rho_t] \succeq \left \{ \begin{array}{ll} -\frac{\lambda_0}{2 t \lambda_0 + 1} & \mbox{if $\Lambda = 0$} \\ -\frac{\Lambda \lambda_0}{\Lambda e^{-\Lambda t} + 2 x_0 (1-e^{-\Lambda t})} & \mbox{else} \end{array} \right . 
            \end{equation}
            \item If $\lambda_0 = +\infty$, one has
            \begin{equation}
                D^2 u[\rho_t] \succeq \left \{ \begin{array}{ll} -\frac{1}{2 t} & \mbox{if $\Lambda = 0$} \\ -\frac{\Lambda}{2 (1-e^{-\Lambda t})} & \mbox{else} \end{array} \right . 
            \end{equation}
        \end{enumerate}
    \end{theorem}

    We shall first provide a self-contained proof in the smooth case, by relying on a maximum principle argument. By stability of the gradient flow solution, one would then be able to recover the general theorem in the following cases : either $W = 0$, or $\lambda_0 = +\infty$. For instance one can approximate the initial data and the potential $V$ by smooth function, while preserving the bound $D^2 u[\rho_0] \succeq -\lambda_0$. Indeed, one can perform a regularization by kernel of $u[\rho_0]$ to obtain $u^\epsilon_0$ smooth satisfying $D^2 u^\epsilon \succeq -\lambda_0$, a similar regularization for $V$, and then define $\rho_0^\epsilon \propto \exp(u^\epsilon - V^\epsilon)$. In the case $\lambda_0 = +\infty$ and $W \neq 0$, it suffices to perform a regularization of $V,W,\rho_0$ to obtain the result. 

    On the other hand, in the case $W \neq 0$, it is not clear how to perform a regularization \emph{while} preserving the lower bound $D^2 u[\rho_0] \succeq -\lambda_0$, or at least asymptotically (i.e. $D^2 u[\rho_0^\epsilon] \succeq -\lambda_0 - o(1)$. This is because the relation between $\rho_0$ and $u[\rho_0]$ is now highly non-trivial. One can for instance easily construct approximation $\rho_0^\epsilon$ satisfying $D^2 u[\rho_0^\epsilon] \succeq -\lambda_0^\epsilon$ with $\lambda_0^\epsilon$ converging, but it might be the case that the limit is smaller than $\lambda_0$, which won't give us the sharpest estimate possible. This issue will also appears at the level of the JKO scheme. 

    We still chose to give the proof in the smooth case, as it is far more illuminating than the non-regular case, and might be adaptable to other settings. The general case will be obtained as a by-product of the discrete version of the inequality.

    \begin{proof}[Proof in the smooth case]
        We assume that $V,W,\rho_0$ are smooth, with $\rho_0 > 0$. We let $-\lambda_t := \min_{x \in \bb{T}^d, \nu \in \bb{S}^d} D^2 u[\rho_t](x)(\nu,\nu)$, well-defined in $[0,+\infty)$ by compactness of $\bb{T}^d \times \bb{S}^d$. We shall first fix $t > 0$, and we let $(x_t,\nu_t)$ be any points reaching optimality for $\lambda_t$. We argue that, at those points, we have
        \[
            \partial_t u_{t,\nu_t \nu_t}(x_t) \geq 2 \lambda_t^2 - \Lambda \lambda_t
        \]
        From now on, until we have finished the proof of this inequality, we shall drop dependency on $t$ of $\nu_t,x_t$, and assume all computation to be done at the peculiar point $x_t$. 

        Optimality conditions and the semi-convexity bounds first gives the following: 
        \begin{itemize}
            \item \emph{Optimality in $\nu$: } This forces $\nu$ to be an eigenvector of $D^2 u$, and $-\lambda_t$ is then the associated eigenvalue. That is $D^2 u \cdot \nu = \nabla u_\nu = -\lambda_t \nu$. 
            \item \emph{Optimality in $x$: } This gives the first and second order condition $\nabla u_{t,\nu \nu} = 0$ and $D^2 u_{t,\nu \nu} \succeq 0$. 
            \item \emph{Gradient estimate: } By lemma \ref{lemma: gradient_estimate} we also have $|\nabla u|_\infty \leq \lambda_t / 2$. 
        \end{itemize}
        Also deriving twice the equation in direction $\nu$ at the minimum point gives
        \begin{equation}
            \partial_t u_{t,\nu \nu} = \underset{(U)}{\underbrace{\Delta u_{t,\nu \nu} + 2 \nabla u_t \cdot \nabla u_{t,\nu \nu} + 2 |\nabla u_{t,\nu}|^2}} - \underset{(V)}{\underbrace{[\nabla V \cdot \nabla u]_{\nu \nu}}} - \underset{(W)}{\underbrace{\cl{R}(\nabla W, \nabla u_t, \rho_t)_{\nu \nu}}}
        \end{equation}
        We estimate each of those three terms separately.

        \begin{enumerate}
            \item[(U)] By optimality in $x$ we have $\Delta u_{t,\nu \nu} \geq 0$ and $\nabla u_t \cdot \nabla u_{t,\nu \nu} = 0$. Furthermore, optimality in $\nu$ gives $2 |\nabla u_{t,\nu}|^2 = 2 \lambda_t^2$. Hence we get $(U) \geq 2 \lambda_t^2$. 
            
            \item[(V)] We extends the derivatives to get
            \[
                (V) = \nabla V_{\nu \nu} \cdot \nabla u + 2 \nabla V_{\nu} \cdot \nabla u_{t,\nu} + \nabla V \cdot \nabla u_{t,\nu \nu}
            \]
            by optimality in $x$, we have $\nabla V \cdot \nabla u_{t,\nu \nu} = 0$, furthermore using the Lipschitz bound on $V$, and the duality between the $1$ and $\infty$ norms giving $v \cdot w \leq |v|_1 \: |w|_\infty$ to get
            \[
                \nabla V_{\nu \nu} \cdot \nabla u \leq |\nabla \partial_{\nu \nu} V|_1 |\nabla u|_\infty \leq \frac{1}{2} L_V \lambda_t
            \]
            and using optimality in $\nu$ we have
            \[
                2 \nabla V_\nu \cdot \nabla u_{t,\nu} = -2 \lambda_t D^2 V[\nu,\nu] \leq 2 \lambda_t \lambda_V
            \]
            as $\lambda_t \geq 0$ since there is no strictly convex periodic functions. Therefore we obtain
            \[
                (V) \leq \left ( \frac{1}{2} L_V + 2 \lambda_V \right ) \lambda_t
            \]

            \item[(W)] Expending again the derivative we obtain
            \begin{align*}
                (W) &= \int_{\bb{T}^d} \nabla W_{\nu \nu}(x-y) \cdot (\nabla u_t(x) - \nabla u_t(y)) \diff \rho_t(y) \\
                &+ 2 [\nabla W_\nu * \rho_t] \cdot \nabla u_{t,\nu} + [\nabla W * \rho_t] \cdot \nabla u_{t,\nu \nu}
            \end{align*}
            again the last term is equal to $0$. We can bound the first term as
            \[
                \int_{\bb{T}^d} \nabla \partial_{\nu \nu} W(x-y) \cdot (\nabla u_t(x) - \nabla u_t(y)) \diff \rho_t(y) \leq \int_{\bb{T}^d} 2 |\nabla \partial_{\nu \nu} W|_1 |\nabla u_t|_\infty \diff \rho_t \leq L_W \lambda_t 
            \]
            and the second term as
            \[
                2 [\nabla W_\nu * \rho_t] \cdot \nabla u_{t,\nu} = -2 \lambda_t \int_{\bb{T}^d} D^2 W(x-y)[\nu,\nu] \diff \rho_t(y) \leq 2 \lambda_t \lambda_W
            \]
            Hence we obtain
            \[
                (W) \leq (L_W + 2 \lambda_W) \lambda_t
            \]
        \end{enumerate}
        Combining those three inequalities with the equation give the inequality.

        Now using the envelope theorem, and the smoothness of $u$, combined with the compactness of $\bb{T}^d \times \bb{S}^d$, we
        deduce that the function $t \to \lambda_t$ is locally absolutely continuous on $[0,+\infty)$, and for any measurable selection $t \to (x_t,\nu_t)$ of optimizers, one has $-\dot{\lambda}_t = \partial_t u_{t,\nu_t \nu_t}(x_t)$ a.e. Combining this with the previous inequality gives that
        \[
            -\dot{\lambda_t} \geq 2 \lambda_t^2 - \Lambda \lambda_t
        \]
        for a.a. $t \in (0,+\infty)$. We then conclude using the Grönwall lemma for logistic equation \ref{lemma: Logistic_Gronwall_Lemma} below. 
    \end{proof}

    \begin{lemma}[Grönwall lemma for logistic equation] \label{lemma: Logistic_Gronwall_Lemma}
        Let $t \in [0,+\infty) \to x_t \in [0,+\infty)$ be a locally absolutely continuous curve such that $\dot{x}_t \leq \Lambda x_t - 2 x_t^2$ for a.a. $t \in (0,+\infty)$. Then one has
        \begin{equation}
            x_t \leq \left \{ \begin{array}{ll} \frac{x_0}{2 t x_0 + 1} & \mbox{if $\Lambda = 0$} \\
            \frac{\Lambda x_0}{\Lambda e^{-\Lambda t} + 2 x_0 (1-e^{-\Lambda t})} & \mbox{else} \end{array} \right . 
        \end{equation}
    \end{lemma}

    \begin{proof}
        If $x_t > 0$ for all $t \geq 0$, then we can divide by $x_t^2$ and obtain the inequality, for $y_t = x_t^{-1}$, $\dot{y}_t \geq 2 - \Lambda y_t$. Multiplying by $e^{\Lambda t}$ and integrating gives $e^{\Lambda t} y_t - y_0 \geq \frac{2}{\Lambda} (e^{\Lambda t} - 1)$ in the case $\Lambda \neq 0$, and $\geq 2 t$ else. Algebraic manipulations then gives the result. In the general case, we consider $x_t^\epsilon := x_t + \epsilon$, which is then strictly positive, and satisfies $\dot{x}_t^\epsilon \leq (\Lambda + 4 \epsilon) x_t^\epsilon - 2 (x_t^\epsilon)^2$, we can then proceed as above, and letting $\epsilon \to 0$ gives the result. 
    \end{proof}

    \begin{remark}
        With the same type of reasoning, one should be able to obtain semi-convexity estimates for a class of non-local Hamilton-Jacobi equations of the form \\ $\partial_t u = \epsilon \Delta u + \int H(x,y;\nabla u(x),\nabla u(y)) \dd{\eta_t(y)}$ for a class of non-linearity $H$. 
    \end{remark}

    The proof in the non-smooth case will rely on the following stability of semi-convexity bounds for the pressure variable.

    \begin{lemma}[Stability of Semi-Convexity bounds] \label{lemma: Stability_Semi_Convexity_Pressure}
        Let $\rho_n \to \rho$ in $\bb{W}_2$, and suppose that $D^2 u[\rho_n] \succeq -\lambda_0^n > -\infty$ with $\lambda_0^n \to \lambda_0$. Then one has $D^2 u[\rho] \succeq -\lambda_0$.  
    \end{lemma}

    \begin{proof}
        The lower bound $D^2 u[\rho_n] \succeq -\lambda_0^n \succeq -\sup_n \lambda_0^n$ implies uniform $L^\infty$ bounds from above and below for $\rho_n$, and uniform Lipschitz bounds (see the proof of proposition \ref{prop: Quantitative_Lipschitz_L_Infty_Bounds}). By Arzela-Ascoli, this shows that $\rho$ is in fact Lipschitz continuous, uniformly bounded from above and below, and $\rho_n \to \rho$ uniformly. In particular, we deduce that $u[\rho_n] \to u[\rho]$ uniformly. Since semi-convexity bounds are stable by uniform limits, we conclude. 
    \end{proof}

    \begin{proof}[Proof in the non-smooth case]
        Suppose first that $\cl{F}[\rho_0] < +\infty$, then if $\rho^\tau$ is a JKO flow starting from $\rho_0$, we have $\rho^\tau_t \to \rho_t$ in $\bb{W}_2$ by theorem \ref{thm: JKO_converges}. Using the asymptotic Li-Yau-Hamilton for the JKO scheme \ref{thm: Asymptotic_Li_Yau_Hamilton_Estimate} and the stability lemma \ref{lemma: Stability_Semi_Convexity_Pressure}, we then obtain the result in this case.
        Now suppose that $\rho_0$ is any initial data. If $D^2 u[\rho_0] \succeq -\lambda_0$ with $\lambda_0 < +\infty$, then we must have $\cl{F}[\rho_0] < +\infty$ and the result is true. Otherwise, we approximate $\rho_0$ by a sequence of initial data $\rho_0^n$ in the effective domain of $\cl{F}$ (which is dense in $\cl{P}(\bb{T}^d)$). Since gradient flow solutions are stable under perturbation of the initial data (\cite{AGS} theorem 11.2.1.), we have $\rho_t^n \to \rho_t$ in $\bb{W}_2$, and we conclude again using again lemma \ref{lemma: Stability_Semi_Convexity_Pressure}. 
    \end{proof}    

\subsection{First consequence: Lipschitz and $L^\infty$-bounds}

    As explained in the introduction, the Li-Yau-Hamilton estimate provides quantitative regularization effects for the flow. We give two instances of such results:
    \begin{enumerate}
        \item $\rho$ is locally (in positive time) uniformly Lipschitz, with explicit constant. 
        \item One has a quantitative version of the strong-maximum principle: $\rho$ is bounded away from $0$ and $+\infty$ with explicit constant.  
    \end{enumerate}

    Lipschitz bounds follows immediately from the semi-convexity to Lipschitz estimate given by lemma \ref{lemma: gradient_estimate}. We can also derive $L^\infty$ bounds. 

    \begin{proposition}[Quantitative Lipschitz and $L^\infty$ Bounds] \label{prop: Quantitative_Lipschitz_L_Infty_Bounds}
        Let $\rho$ be any solution to the Aggregation-Diffusion equation starting from $\rho_0$ satisfying $D^2 u[\rho_0] \succeq -\lambda_0$. Then 
        \begin{equation}
            |\nabla u[\rho_t]|_\infty \leq \frac{1}{2} E_t^{\lambda_0}
        \end{equation}
        where $E_t^{\lambda_0}$ is the function appearing in the right-hand side of the Li-Yau-Hamilton inequality, depending on $\lambda_0 \in [0,+\infty]$. In particular
        \begin{equation}
            |\nabla \log \rho_t|_\infty \leq |\nabla V|_\infty + |\nabla W|_\infty + \frac{1}{2} E_t^{\lambda_0} =: L_t^{\lambda_0}
        \end{equation}
        Furthermore, one has the following $L^\infty$ bound.
        \begin{equation}
            \exp(-\frac{d \sqrt{d}}{2} L_t^{\lambda_0}) \leq \rho \leq \exp(\frac{d \sqrt{d}}{2} L_t^{\lambda_0})
        \end{equation}
    \end{proposition}

    \begin{proof}
        The last estimate follows from the inequality $|v|_2 \leq \sqrt{d} |v|_\infty$ which gives that any function satisfying $|\nabla f|_\infty \leq L$ is in fact $\sqrt{d} L$ Lipschitz, combined with the following lemma. 
    \end{proof}

    \begin{lemma}[Lipschitz to $L^\infty$ Bounds] \label{lemma: Log-Lip_to_L_infty}
        Let $\eta \in \cl{P}(\bb{T}^d)$ be positive, and such that $\log \eta$ is $L$-Lipschitz. Then one has
        \begin{equation}
            e^{-\frac{d}{2} L} \leq \eta \leq e^{\frac{d}{2} L}
        \end{equation}
    \end{lemma}

    \begin{proof}
        Since the torus has diameter $\frac{d}{2}$, the Lipschitz bound gives that for all $(x,y) \in \bb{T}^d$, one has
        \[
            e^{-\frac{d}{2} L} \eta(y) \leq \eta(x) \leq e^{\frac{d}{2} L} \eta(y)
        \]
        then integrating on $y$ gives the estimate.
    \end{proof}

    \begin{remark}
        This is to be compared, when $W = 0$ case, to the classical Lipschitz estimate $|\nabla u[\rho_t]|_2 \leq |\nabla u[\rho_0]|_2 e^{\lambda_V t}$ which holds under the less stringent assumption that $V$ is semi-convex (but under more regularity on the initial data). We observe that the Lipschitz estimate we derive using the Li-Yau-Inequality is worse for small time $t$ (if we only assume that $\rho_0$ is such that $u[\rho_0]$ is Lipschitz), but better in large time. 
    
        The classical Lipschitz estimate, on the other hand, holds for quite general domain (for instance regular convex domains). It can also be extended to the JKO scheme : this was first done by \cite{Lee_Lipschitz} by P.W. Lee in the case of the torus, and then extended by Ferrari and Santambrogio in \cite{Lip_JKO_Filippo} for a generalization to any convex domain. An extension to general modulus of continuity was also proved by Caillet and Santambrogio \cite{Caillet_Continuity} for solutions to a class of doubly non-linear diffusion equations (both at the continuous and JKO level).
    \end{remark}
    
\subsection{Quantitative Harnack Inequality}

    In parabolic theory, Harnack inequalities states that the maximal value on some cylinder of the solution is controlled by the minimal value on some smaller cylinder. This is a fundamental tool in the theory, providing Kernel estimates for linear equation, and used to derive Hölder regularity. It is well-known that integration of the Li-Yau inequality along geodesics provides a quantitative version of the Harnack inequality for the heat equation \cite{Li-Yau}. The same technique can be used in our setting to derive a quantitative Harnack inequality for the Granular-Medium equation. 
    
    This Harnack inequality will be quantified using the following Lagrangian. 

    \begin{definition}[Lagrangian Associated with a Solution]
        Let $\rho$ be a gradient flow solution to the Aggregation-Diffusion equation. To this solution we associate a Lagrangian defined for any $t \geq 0$, $x \in \bb{T}^d$ and $p \in \bb{R}^d$ by
        \begin{equation} \label{eq: Lagrangien_Def}
            \cl{L}_\rho(p,x,t) = \frac{1}{4} |p + \nabla V(x) + \nabla W * \rho_t(x)|^2
        \end{equation}
        And the associate pseudo-metric defined for $x,y \in \bb{T}^d$, $t,h \geq 0$ by
        \begin{equation} \label{eq: Lagrangien_Cost}
            D_\rho(x,y;t,h) = \inf \left \{ \int_t^{t+h} \cl{L}_\rho(\dot{\gamma}_s,\gamma_s,s) \diff s, \gamma_t = x, \gamma_{t+h} = y \right \}
        \end{equation}
        where we take the infimum over all $\AC_2([t,t+h];\bb{T}^d)$ curves. 
    \end{definition}

    \begin{theorem}[Quantitative Harnack Inequality] \label{thm: Harnack_Ineq}
        Let $\rho$ be any solution to the Aggregation-Diffusion equation. Then for all $t,h > 0$, $x,y \in \bb{T}^d$ one has
        \begin{equation} \label{eq: Harnack_Ineq}
            \rho_t(x) \leq \rho_{t+h}(y) \left ( \frac{e^{\Lambda (t+h)} - 1}{e^{\Lambda t}-1} \right )^d \exp(D_\rho(x,y;t,h))
        \end{equation}
    \end{theorem}

    \begin{proof}
        Let $\gamma \in C^1([t,t+h];\bb{T}^d)$ be such that $\gamma_t = x$ and $\gamma_{t+h} = y$. Set $p_t := \log \rho_t$ for $t > 0$. Notice that $p$ is a classical solution on $(0,+\infty) \times \bb{T}^d$ to the equation $\partial_t p_t = \Delta u_t +|\nabla p_t|^2 - \nabla q_t \cdot \nabla p_t$ with $u_t = u[\rho_t]$ and $q_t = \nabla V + \nabla W * \rho_t$. Since $D^2 u_t \succeq -E_t^{\infty,\Lambda} =: -E_t :$ we have $\Delta u_t \geq -d E_t$. Taking the derivative of $p_t$ along $\gamma_t$ we have
        \begin{align*}
            \dv{t} p_t(\gamma_s) &= \partial_t p_t(\gamma_t) + \nabla p_t(\gamma_t) \cdot \dot{\gamma_t} \\
            &\geq -d E_t + |\nabla p_t|^2 + [\dot{\gamma}_t - \nabla q_t] \cdot \nabla p_t \\
            &\geq -d E_t - \frac{1}{4} |\dot{\gamma}_s - \nabla q_s|^2
        \end{align*}
        Therefore we have
        \[
            \log \rho_t(x) \leq \log \rho_{t+h}(y) + d \int_t^{t+h} E_s \diff s + \int_t^{t+h} \cl{L}_\rho(\dot{\gamma}_s,\gamma_s,s) \diff s
        \]
        taking the infimum in $\gamma$, and if we observe that $E_t$ admits $\frac{1}{2} \log (e^{\Lambda t} - 1)$ as primitive. We obtain
        \[
            \log \rho_t(x) \leq \log \rho_{t+h}(y) + \log \left ( \frac{e^{\Lambda (t+h)} - 1}{e^{\Lambda t}-1} \right )^\frac{d}{2} + D_\rho(x,y;t,h)
        \]
        which concludes the proof after taking the exponential. 
    \end{proof}

    One can obtain simpler form of the inequality, using for instance the following bound (which we do not claim to be sharp). 

    \begin{lemma}[Upper bound on the Lagrangian Cost] \label{lemma: Upper_Bound_Lagrangien_Cost}
        For any $x,y \in \bb{T}^d, t,h > 0$ one has, for $B = B_{V,W} = |\nabla V|_2 + |\nabla W|_2$,
        \begin{equation} \label{eq: distance_bounded_by_euclidean}
            D_\rho(x,y,t,h) \leq \left ( \frac{d(x,y)}{2 \sqrt{h}} + \frac{1}{2} \sqrt{h} B \right )^2
        \end{equation}
    \end{lemma}

    \begin{proof}
        Expending the square and bounding the scalar product one has that for any $\epsilon > 0$ 
        \[
            \cl{L}_\rho(p,x,t) \leq \frac{1+\epsilon}{4} |p|^2 + \frac{1 + \epsilon^{-1}}{4} B^2
        \]
        Then integrating against paths and minimizing we obtain
        \[
            D_\rho(x,y;t,h) \leq (1+\epsilon) \frac{d(x,y)^2}{4 h} + \frac{1 + \epsilon^{-1}}{4} h B^2
        \]
        the final inequality is obtained by minimizing over $\epsilon > 0$. 
    \end{proof}

    \begin{remark}
        Note that when $W = 0$, the Lagrangian does not depends on the solution itself. Minimizers of the action functional solves the Newton equation with force $F = -D^2 V \cdot \nabla V = -\frac{1}{2} \nabla |\nabla V|^2$. That is 
        \[
            \ddot{\gamma}_t = -D^2 V[\gamma_t] \cdot \nabla V[\gamma_t]
        \]
        Indeed, by expending the square and using that $\nabla V(\gamma_t) \cdot \dot{\gamma_t}$ is the derivative of $V(\gamma_t)$ one has
        \begin{equation}
            D_\rho(x,y;t,h) = \frac{1}{2} (V(x) - V(y)) + \frac{1}{2} \inf_{\gamma} \int_t^{t+h} (E_c[\gamma_s] + E_p[\gamma_s]) \diff s
        \end{equation}
        where $E_c$ is the kinetic energy and $E_p$ the potential energy associated with the force $F$, i.e. with potential $\frac{1}{2} |\nabla V|^2$. 
    \end{remark}

    \begin{remark}
        By using the bound on the Lagrangian cost $D_\rho$ one can obtain Parabolic Harnack inequalities of the form
        \[
            \sup_K \rho(t-h,\cdot) \leq C_{K,h,t_0} \inf_K \rho(t,\cdot)
        \]
        for any $t \in (t_0,T)$, $h > 0$ and $K$ compact, with explicit constant $C$ not depending on $\rho$. 
    \end{remark}

\section{Preliminaries on Periodic Optimal Transport and on the JKO Scheme}

    \subsection{Periodic Optimal Transport}

    We recall the basics properties of optimal transportation on the torus. For the general theory of optimal transport, we invite the reader to consult the classical monographs by Santambrogio, Villani, Ambrosio, Gigli and Savaré \cite{OT_Filippo} \cite{Old_And_New}, \cite{Topics_OT} and \cite{AGS}. 

    \begin{definition}[Wasserstein $2$-Distance]
        Let $\mu,\nu \in \cl{P}(\bb{T}^d)$, a transport plan between $\mu$ and $\nu$ is a probability measure on $\bb{T}^d \times \bb{T}^d = \bb{T}^{2d}$ with first and second marginals given by $\mu,\nu$. The set of these transport plans is denoted by $\Pi(\mu,\nu)$. 
        The square Wasserstein distance is defined by 
        \begin{equation}
            W_2(\mu,\nu)^2 := \inf_{\gamma \in \Pi(\mu,\nu)} \int_{\bb{T}^d \times \bb{T}^d} d(x,y)^2 \dd{\gamma(x,y)}
        \end{equation}
    \end{definition}

    It is a well-known fact that the infimum is always attained and the that Wasserstein distance is a genuine distance on $\cl{P}(\bb{T}^d)$ which metrizes the narrow topology (i.e. in duality with $C(\bb{T}^d)$). An important result in the theory is the following dual formulation, called Kantorovich formulation.

    \begin{proposition}[Kantorovich dual formulation]
        For any $\mu,\nu \in \cl{P}(\bb{T}^d)$ one has 
        \[ \frac{1}{2} W_2^2(\mu,\nu) = \sup_{\psi,\phi} \int_{\bb{T}^d} \psi \diff \mu + \int_{\bb{T}^d} \phi \diff \nu \]
        where the supremum is taken over all continuous periodic function satisfying $\psi(x) + \phi(y) \leq \frac{1}{2} d(x,y)^2$. Furthermore, the supremum is attained at a pair of $c$-conjugate functions, that is satisfying 
        \begin{align}
            \psi(x) = \phi^c(x) = \inf_y \frac{1}{2} d(x,y)^2 - \phi(y) \quad  \phi(y) = \psi^c(y) = \inf_x \frac{1}{2} d(x,y)^2 - \psi(x) 
        \end{align}
        We call such a pair a pair of Kantorovich potentials from $\mu$ to $\nu$. 
    \end{proposition}

    A function $\psi$ equal to the conjugate of another function $\phi$ is called a $c$-concave function. In our setting, it is not hard to see that this is equivalent to the fact that the function $\psi$, seen as a periodic function over $\bb{R}^d$, is a $1$-concave (i.e. $D^2 \psi \preceq \rm{I}$ weakly). Furthermore, if $\psi,\phi$ is a pair of Kantorovich potentials, and $\gamma$ and optimal transport plan, then the inequality $\psi(x) + \phi(y) \leq \frac{1}{2} d(x,y)^2$ becomes an equality on the support of $\gamma$. 

    This existence result is the basics block for the generalization of Brenier's theorem \cite{Brenier} to Riemannian manifolds. Even if this was solved in full generality by McCann in \cite{McCann}, the case of the torus can be studied independently using the simple structure this space. This has been done by Cordero-Erausquin in \cite{PeriodicOT} (in french) (see also section 1.3.2 of \cite{OT_Filippo} for an english version). 

    \begin{theorem}[Brenier-McCann-Cordero]
        Suppose $\mu \ll \cl{L}^d$, let $(\psi,\phi)$ be a pair of Kantorovich potential from $\mu$ to $\nu$. Then
        \begin{enumerate}
            \item $\psi$ is differentiable $\mu$-a.s. And, defining $T := \id - \nabla \psi : \bb{T}^d \to \bb{R}^d$, then $(\id,T)_\# \mu$ is the unique optimal transport plan between $\mu$ and $\mu$. Furthermore, $\mu$-a.e. $T(x) - x \notin \partial Q + \bb{Z}^d$. We call $T$ the optimal transport map from $\mu$ to $\nu$ (unique $\mu$-a.e. and up to $\bb{Z}^d$-translations). 
            \item If we also have $\nu \ll \cl{L}^d$, then for $S = \id - \nabla \phi$. One $S \circ T = \id \mod \bb{Z}^d$ $\mu$ a.e. 
            \item Furthermore, the Monge-Amp\'ere equation
            \[
                \nu \circ T \, \det(D T) = \mu
            \]
            holds $\mu$-a.e. 
        \end{enumerate}
    \end{theorem}

    Note that as $|\nabla \phi|_\infty \leq 1/2$ a.e. We have $T x - x \in \overline{Q}$ a.e. As $T x - x$ is not in $\partial Q + \bb{Z}^d$ for $\mu$ a.a. $x$, this forces $T x - x$ to be in the interior of $Q$ for $\mu$ a.a. $x$. 

    \begin{remark}
        The fact that $T(x) - x \notin \partial Q + \bb{Z}^d$ means that a.s., $T(x)$ is a point of differentiability of $y \to d(x,y)^2$, in other word, $T(x)$ is not in the cut locus at the point $x$. Note that this is a general fact in the theory of optimal transportation in Riemannian manifolds. 
    \end{remark}

    Finally, similarly to the classical case, as proved in \cite{PeriodicOT}, one can apply Caffarelli's regularity theory \cite{Caffarelli} for the Monge-Ampère equation to obtain global regularity of the optimal potential (see also \cite{MongeAmpPeriodic}). 

    \begin{theorem}[Caffarelli's Regularity]
        Suppose that there is $\epsilon > 0$ with $\epsilon \leq \mu,\nu \leq \epsilon^{-1}$. Let $(\psi,\phi)$ be a pair of Kantorovich potentials from $\mu$ to $\nu$. Then
        \begin{itemize}
            \item There exists some $\beta \in (0,1)$ such that $\psi \in C^{1,\beta}(\bb{T}^d)$ and $\psi$ is strictly $1$-concave.
            \item If $\mu,\nu$ are of class $C^{k,\alpha}(\bb{T}^d)$ for some $\alpha \in (0,1)$ and $k \geq 0$, then $\psi$ is of class $C^{k+2,\alpha}(\bb{T}^d)$ and the Monge-Ampère equation holds in the classical sense. 
        \end{itemize}
    \end{theorem}

    We refer the interested reader to \cite{FigalliMongeAmpere} for more discussion about the regularity theory for the Monge-Ampère equation and links to optimal transport. 
    
\subsection{The JKO Scheme}

    For $\rho \in \cl{P}(\bb{T}^d)$ absolutely continuous with respect to the Lebesgue measure, we recall the definition of the energy functional
    \begin{align}
        \cl{F}[\rho] &:= \int_{\bb{T}^d} \rho \log \rho \diff x + \int_{\bb{T}^d} V \diff \rho + \int_{\bb{T}^d} W * \rho \diff \rho = \cl{E}[\rho] + \cl{V}[\rho] + \cl{W}[\rho]
    \end{align}
    sum of a local energy term, a potential energy, and an interaction energy. We also set $\cl{F}[\rho] = +\infty$ whenever $\rho$ does not admit a density. 

    \begin{proposition}[One-Step JKO Scheme]
        Let $\mu \in \cl{P}(\bb{T}^d)$. Then there exists a minimizer to the problem
        \[ 
            \inf_{\mu \in \cl{P}(\bb{T}^d)} \cl{F}[\rho] + \frac{1}{2 \tau} W_2^2(\rho,\mu) 
        \]
         We shall denote by $\Prox_\tau[\rho]$ (the proximal set) the set of all minimizers of the one-step JKO scheme associated to $\mu$.
    \end{proposition}

    The existence part follows easily from the direct method using the l.s.c. of the functional for the narrow topology. 

    \begin{remark}
        If $W = 0$, then by strict convexity of the energy, one always has uniqueness of minimizers. On the other hand, when $W$ is non-zero, this is not true anymore. One can show that there exists $\tau_0 > 0$ depending only on $V,W$ (in fact only on semi-convexity bounds for the potentials) such that minimizers are unique for $\tau < \tau_0$. This relies on a version of geodesic convexity along generalized geodesics adapted to the case of the torus. 
    \end{remark}

    The following proposition is well-known in the theory of the JKO scheme.

    \begin{proposition}
        Let $\eta \in \cl{P}(\bb{T}^d)$, and $\rho \in \Prox_\tau[\eta]$. Let $(\psi,\phi)$ be a pair of Kantorovich potential from $\rho$ to $\eta$. Then
        \begin{enumerate}
            \item One has $\rho > 0$, and there exists a constant $C$ such that 
            \[
                u[\rho] = -\frac{1}{\tau} \psi + C 
            \]
            which can be taken to be zero up to modifying the potentials. In particular, $\rho$ is of class $C^{0,1}(\bb{T}^d)$. 
            \item If $\eta > 0$ is of class $C^{0,1}(\bb{T}^d)$. Then $\rho,u[\rho]$ and $\psi$ are of class $C^{2,\alpha}(\bb{T}^d)$ for any $\alpha < 1$. And the Monge-Ampère equation 
            \begin{equation}
                \det(I - D^2 \psi) \rho(\id - \nabla \psi) = \rho
            \end{equation}
            holds in the classical sense. 
            \item If $\eta > 0$ is of class $C^{2,\alpha}(\bb{T}^d)$ for some $\alpha$, then $\rho$ is of class $C^{2,1}(\bb{T}^d)$ and $u[\rho],\psi,\phi$ are of class $C^{4,\alpha}(\bb{T}^d)$. 
        \end{enumerate}
    \end{proposition}

    \begin{proof}
    \begin{enumerate}
        \item The positivity of $\rho$ and the optimality condition follows from an easy adaptation of the argument, presented in the non-periodic setting and without interaction, in Chapter 8 of \cite{OT_Filippo}. 

        \item By optimality condition, $u[\rho]$ is Lipschitz, hence $\log \rho$ is also Lipschitz as $V,W$ are. Using that $\rho > 0$, we deduce that $\rho$ is also Lipschitz. By Caffarelli's regularity, this implies that $(\psi,\phi)$ are of class $C^{2,\alpha}(\bb{T}^d)$ for all $\alpha < 1$, which in turn implies that $u[\rho]$ and $\rho$ are also of class $C^{2,\alpha}(\bb{T}^d)$ for all $\alpha < 1$. 

        \item We obtained above that $\rho$ is of class $C^{2,\alpha}$, using that $\eta$ is also of this class, we obtain that $(\psi,\phi)$ are of class $C^{4,\alpha}(\bb{T}^d)$, which concludes using that $V,W$ are of class $C^{2,1}(\bb{T}^d)$. 
    \end{enumerate}
    \end{proof}

    \begin{remark}
        Using that $\psi$ is Lipschitz, one can derive the following quantitative estimate using lemma \ref{lemma: Log-Lip_to_L_infty}: there exist two constants $c,C > 0$ depending only on $V,W$ and $\kappa$ depending only on $d$ such that
        \[
            c e^{\frac{-\kappa}{\tau}} \leq \rho \leq C e^{\frac{\kappa}{\tau}}
        \]
    \end{remark}

    \begin{definition}[JKO flow]
        Let $\rho_0 \in \cl{P}(\bb{T}^d)$, a JKO flow starting from $\rho_0$ is any sequence of measure $(\rho_k^\tau)$ with $\rho_0 \in \cl{P}(\bb{T}^d)$ and satisfying $\rho_{k+1}^\tau \in \Prox_\tau[\rho_k^\tau]$. We shall let $\rho^\tau_\cdot$ be the piecewise constant interpolation of the values of $(\rho_k^\tau)_{k \geq 0}$ (i.e. constant equal to $\rho_k^\tau$ on $[k \tau, (k+1) \tau)$). 
    \end{definition}

    The following theorem is the fundamental result in the theory of the JKO scheme, originally proved by Jordan, Kinderlehrer and Otto in \cite{JKO}, stating convergence of the JKO scheme to the continuous equation. The case of the torus can be easily obtained by modifications of the argument of chapter $8$ of \cite{OT_Filippo} 

    \begin{theorem}[Convergence to the continuous equation] \label{thm: JKO_converges}
        Suppose $\cl{F}[\rho_0] < +\infty$, $T > 0$. Then $\rho^\tau$ converges uniformly on $[0,T]$ in $\bb{W}_2$ to a gradient flow solution to the Granular-Media equation starting from $\rho_0$. Furthermore, the limit curve is of class $C^{0,1/2}([0,T];\bb{W}_2)$. 
    \end{theorem}

\section{Asymptotic Li-Yau-Hamilton Estimate for the JKO Scheme}

    The goal of this section is to prove the main result of the paper: namely the asymptotic Li-Yau-Hamilton estimate for the JKO scheme associated to the Granular-Medium equation. 

\begin{theorem}[Asymptotic Li-Yau-Hamilton Estimate] \label{thm: Asymptotic_Li_Yau_Hamilton_Estimate}
    Let $(\rho_k^\tau)_{k \geq 0}$ be any JKO flow starting from $\rho_0$, and suppose that $D^2 u[\rho_0] \succeq -\lambda_0$ with $\lambda_0 \in (0,+\infty]$. Then
    
    \begin{enumerate}
        \item If $\lambda_0 < +\infty$, then for all $\epsilon > 0$, there exists $\tau(\epsilon) > 0$ (depending only on $\lambda_0,V,W$) such that for all $t \geq 0$ one has
        \begin{equation}
            D^2 u[\rho_t^\tau] \succeq \left \{ \begin{array}{ll} -(1+\epsilon) \frac{\lambda_0}{2 \lambda_0 t + 1} & \mbox{if $\Lambda = 0$} \\
            -(1+\epsilon) \frac{\Lambda \lambda_0}{\Lambda e^{-\Lambda t} + 2 \lambda_0 (1- e^{-\Lambda t})} & \mbox{else} \end{array} \right . 
        \end{equation}
        \item If $\lambda_0 = +\infty$, then for all $t_0 > 0$, and $\epsilon > 0$ small enough (such that no divisions by zero occur), there exists $\tau(\epsilon,t_0) > 0$ such that for all $\tau < \tau(\epsilon,t_0)$ and $t \geq t_0$ one has
        \begin{equation}
            D^2 u[\rho_t^\tau] \succeq \left \{ \begin{array}{ll} -(1+\epsilon) \frac{1}{2 t} & \mbox{if $\Lambda = 0$} \\
            -(1+\epsilon) \frac{\Lambda}{2 (1- e^{-\Lambda t})} & \mbox{else} \end{array} \right . 
        \end{equation}
    \end{enumerate}
\end{theorem}

Before going into the proof, let us briefly explain the different steps. 
\begin{enumerate}
    \item Following P.W. Lee's strategy. We first show a one-step improvement of semi-convexity along the JKO scheme. That is, if one starts from a measure $\eta$ satisfying $D^2 u[\eta] \succeq -\lambda_0$, then if $\rho$ is obtained from $\eta$ after one-step of the JKO scheme, one has $D^2 u[\rho] \succeq -\lambda_1$ with $\lambda_1$ is controlled by $\lambda_0$ through an inequality of the form $G[\tau \lambda_1,\tau] \leq \tau \lambda_0$ for an explicit function $G$ depending only on $\tau,V,W$. 
    
    \item The second step is to iteratively use the previous one-step estimate, relying on a discrete comparison principle, to obtain a lower bound with respect to a sequence $E_k^\tau$ depending only on the initial data on the initial data and on $\tau$, i.e. to show that $D^2 u[\rho_k^\tau] \succeq -\frac{1}{\tau} E_k^\tau$ for any JKO flow, and satisfying $G[E_{k+1}^\tau,\tau] = E_k^\tau$ (that is replacing the inequality in the previous step by an equality). 
    
    \item The final step is to study the asymptotic behavior of the sequence $E_k^\tau$ in the regime $\tau \to 0$ and $k \tau \sim t$. The argument is based on a linearization of $G$ and a comparison with the solution to the ODE solving the linearized problem.  
\end{enumerate}

Before going to the proof, we shall also note that the estimate will take a more precise quantitative form in the case of the heat equation, i.e. when $V$ and $W$ are zeros, see proposition \ref{prop: Asymptotic_Estimate_Heat} for the precise statement in this case. 

\subsection{One-Step Improvement of Semi-Convexity}

    We start by proving the one-step improvement of semi-convexity along the JKO scheme. To state the estimate, we define for $\tau > 0$ and any $E \in [0,1)$ the function
    \begin{equation}
        G[E,\tau] := \frac{E}{(1-E)^2}(1-\tau (2\lambda^* + L^*) + \tau (\lambda^* + L^*) E)
    \end{equation}
    we also extend it to $E = 1$ by the value $+\infty$ (which is consistent with the limit as $E \to 1$ for $\tau$ small enough). 
    
    The one-step improvement then takes the following form
    
    \begin{theorem}[One-Step improvement of Semi-Convexity] \label{thm: One-Step_Improvement}
        Suppose $D^2 u[\eta] \succeq -\lambda_0$ with $\lambda_0 \in [0,+\infty]$. Let $-\lambda_1$ be the minimal possible eigenvalue of $D^2 u[\rho]$. Then:
        \begin{itemize}
            \item If $\lambda_0 < +\infty$ one has $\tau \lambda_1 < 1$ and
            \begin{equation}
                G[\tau \lambda_1,\tau] \leq \tau \lambda_0
            \end{equation}
            \item Else, if $\lambda_0 = +\infty$, the same holds but with $\tau \lambda_1 \leq 1$. 
        \end{itemize}
    \end{theorem}

    As in the continuous case, we shall provide an incomplete, but more illuminating, proof assuming more regularity on the initial data (more precisely, we shall assume that $\eta$ is strictly positive and of class $C^{2,\alpha}(\bb{T}^d)$). If we have $W = 0$, one can easily deduce the general case by approximation. Similarly, if we consider the case $\lambda_0 = +\infty$ and one is only interested in the asymptotic estimate, one can first do two iterations of the scheme, in order to obtain the $C^{2,\alpha}$ regularity and the positivity, and then iterate starting from $k = 2$, or alternatively proceed by approximation. But in the case $W \neq 0$ and $\lambda_0 < +\infty$, we encounter the same issue as in the continuous case: we do not know how to approximate $\eta$ by regular densities while preserving, at least asymptotically, the bound $D^2 u[\eta] \geq -\lambda_0$. 

    The proof in the general case is thus postponed to the appendix, and makes use of a classical strategy in similar problem: replacing second order quantities by finite-differences. 

    \begin{proof}[Proof in the regular case]
        We assume that $\eta \in C^{2,\alpha}(\bb{T}^d)$, $\eta > 0$. We recall that, under these hypotheses, that $\rho$ is of class $C^{2,\alpha}(\bb{T}^d)$ and the Kantorovich potentials $(\psi,\phi)$ are of class $C^{4,\alpha}(\bb{T}^d)$. Furthermore, $\psi,\phi$ are strictly $1$-concave, which gives $\tau D^2 u[\rho] \succ -1$, hence $\tau \lambda_1 < 1$. We let $p := \log \rho$ and $q := \log \eta$. Furthermore we set $v := \frac{1}{2} |x|^2 - \phi$. Taking the logarithm of the Monge-Ampère equation gives
        \[
            \log \det D^2 v = \log \eta - \log \rho(\nabla v) = q - p(\nabla v)
        \]
        we divide the proof into several steps.
        \begin{enumerate}
            \item \emph{Second-Order Derivative of Monge-Ampère equation: } Consider a direction $\nu \in \bb{S}^d$, Since $D^2 v \succ 0$, we can derive twice the Monge-Ampère equation in direction $\nu$, which gives
            \[
                \Tr [D^2 v]^{-1} D^2 v_{\nu \nu} - \Tr |[D^2 v]^{-1} D^2 v_\nu|^2 
            \]
            \[
                = q_{\nu \nu} - D^2 p(\nabla v)[\nabla v_\nu,\nabla v_\nu] - \nabla p(\nabla v) \cdot \nabla v_{\nu \nu}
            \]
            
            \item \emph{Maximum Principle: } Consider a point $x \in \bb{T}^d$, and a direction $\nu$ such that $D^2 v(x)[\nu,\nu]$ is maximal, and denote by $M$ the value of this maximum. Alternatively $D^2 \phi(x)[\nu,\nu]$ is minimal as $D^2 v = I - D^2 \phi$, in particular, as $\phi$ is periodic, we must have $M \geq 1$. From now on, all computations shall be carried at this particular point $x$. We have the following optimality conditions
            \begin{itemize}
                \item As first order condition on $\nu$ we get that $\nabla v_\nu = D^2 v \cdot \nu = M \nu$, that is $\nu$ is an eigenvector of $D^2 v$. 
                \item  For optimality of $x$, we get the first and second order condition $\nabla v_{\nu \nu} = 0$ and $D^2 v_{\nu \nu} \preceq 0$. 
                \item We use the semi-convexity to gradient estimate of lemma \ref{lemma: gradient_estimate}. As $\phi$ satisfies $D^2 \phi \succeq 1-M$, we must have $|\nabla \phi|_\infty \leq \frac{1}{2}(M-1)$, hence $|x-\nabla v|_\infty \leq \frac{1}{2}(M-1)$. 
            \end{itemize}
            
            If we plug this into the second derivative of Monge-Ampère equation, we see that the left side is then non-positive, as $D^2 v$ is a positive definite matrix. Furthermore, the last term vanish and the second term is equal to $M^2 p_{\nu \nu}(\nabla v)$. 
            
            Therefore we obtain 
            \[
                q_{\nu \nu} \leq M^2 p_{\nu \nu}(\nabla v)
            \]
            we also observe that 
            \[
                q = u[\eta] - V - W * \eta \qquad p = u[\rho] - V - W * \rho
            \]
            so that we obtain
            \begin{equation} \label{eq: Optimality_Equation}
                -\lambda_0 \leq u[\eta]_{\nu \nu} \leq M^2 u[\rho]_{\nu \nu}(\nabla v) + \underset{(V)}{\underbrace{V_{\nu \nu} - M^2 V_{\nu \nu}(\nabla v)}} + \underset{(W)}{\underbrace{W_{\nu \nu} * \eta - M^2 W_{\nu \nu} * \rho(\nabla v)}}
            \end{equation}

            \item \emph{Estimating the (V) Term: } We write down 
            \[ 
                (V) = V_{\nu \nu} - V_{\nu \nu}(\nabla v) + (1-M^2) V_{\nu \nu}(\nabla v) 
            \]
            By the Lipschitz estimate on $V$ we have
            \[
                V_{\nu \nu} - V_{\nu \nu}(\nabla v) \leq L_V |x - \nabla v|_\infty \leq \frac{1}{2} L_V (M-1)
            \]
            Furthermore since $1-M^2 \leq 0$ by $M \geq 1$ the last term is bounded above by $\lambda_V (M^2 - 1)$. Hence 
            \[ 
                (V) \leq \lambda_V (M^2 - 1) + \frac{1}{2} L_V (M-1)
            \]
            
            \item \emph{Estimating the (W) Term: } Similarly we have
            \[
                (W) = W_{\nu \nu} * \eta - W_{\nu \nu} * \rho(\nabla v) + (1-M^2) W_{\nu \nu} * \rho(\nabla v)
            \]
            we can again bound the last term by $(M^2-1) \lambda_W$. For the first term will again use the Lipschitz property, but one has to be careful due to the convolution. Using that $\nabla v$ pushes forward $\eta$ to $\rho$ we have
            \[
                W_{\nu \nu} * \rho(\nabla v) = \int W_{\nu \nu}(\nabla v(x) - y) \rho(\diff y) = \int W_{\nu \nu}(\nabla v(x) - \nabla v(y)) \eta(\diff y)
            \]
            Hence
            \begin{align*}
                W_{\nu \nu} * \eta - W_{\nu \nu} * \rho(\nabla v) &= \int [W_{\nu \nu}(x-y) - W_{\nu \nu}(\nabla v(x) - \nabla v(y))] \eta(\diff y) \\
                &\leq L_W \int |\nabla v(x) - x - (\nabla v(y) - y)|_\infty \eta(\diff y) \\
                &\leq L_W (M-1)
            \end{align*}
            Combining these estimates we obtain
            \[
                (W) \leq \lambda_V (M^2 - 1) + L_W (M-1)
            \]

            \item \emph{Relating $M$ to $\lambda_1$: } We now relate $M$ to $\lambda_1$. Since if $u = \frac{1}{2} |x|^2 - \psi$ we have $\nabla u(\nabla v) = \rm{id}$ mod $\bb{Z}^d$ (as $\nabla u$ is the transport map from $\rho$ to $\eta$, and $\nabla v$ the reverse one), by continuity there must be some universal $n \in \bb{Z}^d$ such that $\nabla u(\nabla v) = \rm{id} + n$.  Hence we obtain $\nabla v + \tau \nabla u[\rho](\nabla v) = \rm{id} + n$. Differentiating gives $[D^2 v]^{-1}-I= \tau D^2 u[\rho](\nabla v)$. But then since $\nabla v$ is a diffeomorphism, maximizing the eigenvalues of $D^2 v$ at the point $x$ in direction $\nu$ is the same as minimizing the eigenvalues of $D^2 u[\rho]$ at the point $\nabla v(x)$ in direction $\nu$. Hence we obtain
            \[
                -\tau \lambda_1 = \tau u[\rho]_{\nu \nu}(\nabla v) = \frac{1-M}{M}
            \]
            Hence $M = \frac{1}{1 - \tau \lambda_1}$.

            \item \emph{Conclusion: } Combining the three estimates, we have
            \[
                -\lambda_0 \leq M^2 \frac{1-M}{\tau M} + \lambda^* (M^2 - 1) + L^* (M-1) 
            \]
            if we multiply by $\tau$, and replace $M$ by $\frac{1}{1 - \tau \lambda_1}$, algebraic manipulations then give the identity
            \[
                G[\tau \lambda_1,\tau] \leq \tau \lambda_0
            \]
        \end{enumerate}
    \end{proof}

\subsection{Discrete Comparison Principle}

    We know perform the second step of the proof, that is, we derive a universal lower-bound, depending only on the semi-convexity of the initial data and of the time step $\tau$, for the semi-convexity along the JKO flow. This is based on the following definition.

    \begin{definition}[Comparison sequence] \label{def: Comparison_Sequence}
        Let $\lambda_0 \in [0,+\infty]$ and suppose that $\tau < \tau^*$ with 
        \begin{equation} \label{eq: Tau_Star}
            \tau^* := \min \left ( \frac{1}{\lambda^* + L^*}, \frac{2}{3 \lambda^* + L^*} \right ) \in (0,+\infty]
        \end{equation}
        Then there exists a unique sequence satisfying
        \begin{equation} \label{eq: Comparison_Serquence}
            \left \{ \begin{array}{ll}
                E_0 = \tau \lambda_0 & \\
                E_k^\tau \in [0,1] & \forall k \geq 1 \\
                G[E_{k+1}^\tau,\tau] = E_k^\tau & \forall k \geq 0
            \end{array} \right .
        \end{equation}
        we call this sequence the comparison sequence starting from $\lambda_0$ (at time step $\tau$). We shall also write $t \ \to E_t^\tau$ for the piecewise constant interpolation (with time step $\tau$) of the values of $(\frac{1}{\tau} E_k^\tau)_{k \geq 0}$. 
    \end{definition}

    The discrete comparison sequence then takes the following form.

    \begin{lemma}[Discrete Comparison Principle] \label{lemma: Discrete_Comparison_Principle}
        Suppose that $\tau < \tau^*$, and that $D^2 u[\rho_0] \succeq -\lambda_0$. Then if $E^\tau$ is the comparison sequence starting from $\lambda_0$, then for any JKO flow $(\rho_k^\tau)_{k \geq 0}$ starting from $\rho_0$ we have
        \begin{equation} \label{eq: Ineq_Discrete_Comparison_Principle}
            D^2 u[\rho_k^\tau] \succeq -\frac{1}{\tau} E_k^\tau \qquad \forall k \geq 0
        \end{equation}
    \end{lemma}

    \begin{proof}[Proof of existence and of the discrete comparison principle]
        We can compute that
        \[
            \partial_E G[E,\tau] = \frac{(1 - 2 \tau \lambda^*) E + 1 - \tau(\lambda^* + L^*)}{(1-E)^3}
        \]
        the numerator takes values, when $E \in [0,1)$, between $2 - \tau (3 \lambda^* + L^*)$ and $1- \tau (\lambda^* + L^*)$. Therefore for $\tau < \tau^*$, both of them are strictly positive and one deduce that $\partial_E G[E,\tau] > 0$ on $[0,1)$, hence $G$ defines an increasing diffeomorphism from $[0,1)$ to $[0,+\infty)$, which shows that the comparison sequence is uniquely well-defined. 

        Iterating the one-step estimate, one the obtain a sequence $\lambda_k^\tau$ such that $D^2 u[\rho_k^\tau] \succeq -\lambda_k^\tau$, $\tau \lambda_k^\tau \in [0,1]$ for all $k \geq 1$, and $G[\tau \lambda_{k+1}^\tau,\tau] \leq \tau \lambda_k^\tau$. Then one easily obtain the inequality $\tau \lambda_k^\tau \leq E_k^\tau$ by induction, indeed, applying the inequality at step $k$ gives $G[\tau \lambda_{k+1}^\tau,\tau] \leq \tau \lambda_k^\tau \leq E_k^\tau = G[E_{k+1}^\tau,\tau]$, and hence $\tau \lambda_{k+1}^\tau \leq E_{k+1}^\tau$ as $G[\cdot,\tau]$ is increasing. 
    \end{proof}

\subsection{Asymptotic Estimate for Comparison Sequence}

    The last part of the proof is to show that the comparison sequence satisfies appropriate asymptotic estimate as $\tau \to 0$, $k \tau \sim t$ with $t$ in an adequate set. We shall first study the Heat equation case, i.e. $V=W=0$, since more precise estimates with simpler techniques are available in this case, and then move to the more complicated case of the Fokker-Planck and Granular-Medium equation (i.e. at least one of the potential is non zero). 

    Since the proofs are mostly technical, and do not involve any particularly appealing new ideas, we postpone them to the appendix, and only state here the results. 

    \begin{proposition}[Heat equation Case] \label{prop: Asymptotic_Estimate_Heat}
        Suppose $V=W = 0$. Then
        \begin{enumerate}
            \item If $\lambda_0 = +\infty$, $E_k^\tau$ does not depends on $\tau$, and one has $E_k \sim \frac{1}{2k}$ as $k \to +\infty$. 
            \item If $\lambda_0 < +\infty$, then for all $\tau$ with $\tau \lambda_0 \leq 1$ one has
            \begin{equation}
                \frac{1}{\tau} E_k^\tau \leq \frac{\lambda_0}{k \tau \lambda_0 (2 - \tau \lambda_0) + K} 
            \end{equation}
        \end{enumerate}
    \end{proposition}

    This result is a small improvement of P.W. Lee result, as it does not ask for any regularity on the initial data, is better in the case of regular initial data, and recover the classical estimate in the $\tau \to 0$ limit. We shall note however that most of the ingredients were already present in his work, only a more precise study of the asymptotic of the induction relation was needed to obtain the improved result. 

    When at least on of the potential is non-zero, we obtain a less quantitative estimate. 
    
    \begin{proposition}[Fokker-Planck and Granular-Medium case] \label{prop: Asymptotic_Estimate_Fokker_Planck_Granular_Medium}
        Suppose that $\Lambda > 0$ (i.e. at least $V$ or $W$ is non zero). Let $\tau < \tau^*$, and consider $E_k^\tau$ the comparison sequence starting from $\lambda_0$.
        \begin{enumerate}
            \item If $\lambda_0 < +\infty$, then for all $\epsilon > 0$, we can find $\tau(\epsilon) > 0$ such that for all $t \geq 0$, $\tau < \tau(\epsilon)$ one has
            \begin{equation}
                \frac{1}{\tau} E_t^\tau \leq (1+\epsilon) \frac{\Lambda \lambda_0}{\Lambda e^{-\Lambda t} + 2 \lambda_0 (1- e^{-\Lambda t})}
            \end{equation}
            \item If $\lambda_0 = +\infty$, then for all $\epsilon > 0$ and $t_0 > 0$, we can find $\tau(\epsilon,t_0) > 0$ such that for all $t \geq t_0$, $\tau < \tau(\epsilon,t_0)$ one has
            \begin{equation}
                \frac{1}{\tau} E_t^\tau \leq (1+\epsilon) \frac{\Lambda}{2 (1-e^{-\Lambda t})}
            \end{equation}
        \end{enumerate}
    \end{proposition}

    Using those two previous propositions, we can conclude the proof of theorem \ref{thm: Asymptotic_Li_Yau_Hamilton_Estimate}.

    \begin{proof}[Proof of theorem \ref{thm: Asymptotic_Li_Yau_Hamilton_Estimate}]
        Using the discrete comparison principle \ref{lemma: Discrete_Comparison_Principle}, it suffices to use the asymptotics estimates for the comparison sequence. 
    
        The case of the Fokker-Planck and Granular-Medium is already proved in theorem \ref{prop: Asymptotic_Estimate_Fokker_Planck_Granular_Medium}, therefore we only need to consider the Heat case. 

        First if $\lambda_0 = +\infty$, fix $\epsilon > 0$, we can find $k(\epsilon)$ such that for all $k \geq k(\epsilon)$, one has $E_k \leq \frac{1+\epsilon}{2k}$. Now let $t_0 > 0$, and $t \geq t_0$. Fix $k$ such that $t \in [k \tau, (k+1) \tau)$, so that $E_t^\tau = E_k$. Then if $k \geq k(\epsilon)$, we obtain $E_t^\tau \leq \frac{1+\epsilon}{2 k \tau} \leq \frac{1+\epsilon}{2 t}$. Since we know that $\tau(k+1) \geq t_0$, to have $k \geq k(\epsilon)$ it suffices to ask that $\frac{t_0}{\tau} - \tau \geq k(\epsilon)$, which 
        is true for $\tau$ small enough. 

        Now if $\lambda_0 < +\infty$, consider $t \geq 0$ and $k$ such that $t \in [k \tau, (k+1) \tau)$, so that $E_t^\tau = E_k^\tau$. Assuming that $\tau \lambda_0 \leq 1$, we then have 
        \[ 
            \frac{1}{\tau} E_t^\tau \leq \frac{\lambda_0}{k \tau (2 - \tau \lambda_0) + 1} \leq \frac{\lambda_0}{t \lambda_0 (2-\tau \lambda_0) + 1} = \frac{\lambda_0}{2 t \lambda_0+1} \left ( 1 + \tau \lambda_0 \frac{t \lambda_0}{t \lambda_0 (2-\tau \lambda_0) + 1} \right )
        \]
        Using that $\tau \lambda_0 \leq 1$, and that $x  \to \frac{x}{1+x}$ is bounded by $1$ on $[0,+\infty]$ we obtain
        \[
           E_t^\tau \leq \frac{\lambda_0}{2 t \lambda_0 + 1}(1 + \tau \lambda_0) \leq (1+\epsilon) \frac{\lambda_0}{2 t \lambda_0 + 1}
        \]
        whenever $\tau \leq \frac{\epsilon}{\lambda_0}$. 
    \end{proof}

    \begin{remark}
        Note that in the case of the heat equation, when $\lambda_0 < +\infty$, the proof shows that one can take $\tau(\epsilon) = \frac{\epsilon}{\lambda_0}$.  
    \end{remark}

\section{Applications: Estimates and Local in time convergence}

    As in the classical case, the Li-Yau-Hamilton inequality has several consequences at the level of the JKO scheme : uniform Lipschitz estimate, with the $L^p_{t,\loc} C^{0,\alpha}_x$ convergence as a consequence, boundness of solution and Harnack inequality. We shall also see later on that one can use this estimate to derive $L^2_{t,\loc} H^2_x(\bb{R}_+^* \times \bb{T}^d)$ convergence of the flow in the Fokker-Planck case, for any initial data with finite entropy. 

\subsection{Lipschitz and $L^\infty$-Bounds and $L^p_{t,\rm{loc}} C^{0,\alpha}$-convergence}

    By the semi-convexity bound, and the log-Lipschitz to $L^\infty$ bound given by lemma \ref{lemma: Log-Lip_to_L_infty} we have the following. 

    \begin{proposition}[Universal Bounds] \label{prop: universal_bounds}
        Let $(\rho_k^\tau)_{k \geq 0}$ be any iteration of the JKO flow. Let $t_0 > 0$. Then there exists $\tau_0 > 0$ such that for any $\tau \leq \tau_0$, $(\rho_t^\tau)_{t \geq t_0}$ is uniformly Lipschitz in space and uniformly bounded away from $0$ to $+\infty$ with constant depending only on $\tau_0,t_0,V$ and $W$. 
    \end{proposition}

    \begin{proof}
        Fix $\epsilon$ small enough, so that we can find $\tau_0$ with $D^2 u[\rho_t^\tau] \geq -2 X_t$ where $X_t$ is the function appearing in the asymptotic estimate, for all $\tau < \tau_0$ and $t \geq t_0$. Then the conclusion follows as in the classical case using lemma \ref{lemma: Log-Lip_to_L_infty}. 
    \end{proof}

    As a consequence of the previous estimates, we can improve slightly the weak convergence of the scheme by some compactness Aubin-Lions lemma argument. 

    \begin{proposition}[$L^p_{t,\rm{loc}} C^{0,\alpha}$-convergence]
        Suppose $\scr{F}[\rho_0] < +\infty$, then for any $\alpha < 1$ and $p < +\infty$, $(\rho_t^\tau)_{t \geq 0}$ converges to the unique solution of the Aggregation-Diffusion starting from $\rho_0$ in $L^p([t_0,T];C^{0,\alpha}(\bb{T}^d))$. Furthermore, $\nabla \rho^\tau_t \to \nabla \rho_t$ a.e. for all $t > 0$.   
    \end{proposition}

    \begin{proof}
        We shall use the generalization of Aubin-Lions lemma for piecewise constant function stated below \ref{thm: Aubin_Lions}. We consider $Y$ the dual of Lipschitz function with average $0$ on $\bb{T}^d$. By an argument similar to the one of \cite{L2H2}. We have
        \begin{align*}
            \tau^{-1} ||\rho^\tau-\rho^\tau(\cdot-\tau)||_{L^1([t_0 + \tau,T];Y)} \leq \sum_{k = K}^N W_1(\rho_{k+1}^\tau,\rho_k^\tau) \leq C 
        \end{align*}
        Hence we have the bound on this space. Let $B = C^{0,\alpha}(\bb{T}^d)$, then we can apply the theorem as $W^{1,+\infty}(\bb{T}^d) \hookrightarrow C^{0,\alpha}(\bb{T}^d)$ is compact. Therefore the sequence is relatively compact in \\ $L^p([0,T];C^{0,\alpha}(\bb{T}^d))$. But since this convergence implies weak convergence, we deduce the result. 
    \end{proof}

    \begin{theorem}[\cite{AubinLions} (see also \cite{RossiSavare}) piecewise constant Aubin-Lions lemma] \label{thm: Aubin_Lions}
        Let $X,B,Y$ be Banach spaces with $X \hookrightarrow B$ compactly and $B \hookrightarrow Y$ continuously. Let $(u^\tau)_{\tau \geq 0} \in L^\infty([0,T];X)$ be constant on each $[k \tau, (k+1) \tau]$, and suppose that for some constant $C$ one has
        \begin{equation}
            \tau^{-1} ||u^\tau - u^\tau(\cdot-\tau)||_{L^1([\tau,T];Y)} + ||u^\tau||_{L^\infty([0,T];Y)} \leq C
        \end{equation}
        for all $\tau \ll 1$. Then $(u^\tau)_{\tau}$ is relatively compact in $L^p([0,T];B)$ for all $p < +\infty$. 
    \end{theorem}

\subsection{Harnack Inequality}

    We shall prove a version of the Harnack inequality for the JKO scheme. We shall note that this, in the $\tau \to 0$ limit, does not recover the full Harnack inequality for the Granular-Medium equation. We believe that a more precise study of the argument in the proof might recover the continuous time version. The proof follows closely the proof of P.W. Lee for the heat equation, with some minor modifications. 

    This Harnack estimate takes the following form:

    \begin{theorem}
        Let $t_0,\epsilon > 0$. Then there exists $\tau(\epsilon,t_0)$ depending only on $\epsilon,t_0,V,W$, and a constant $C$ depending only on $t_0,V,W$ and an upper bound on $\epsilon$ such that for all $t \geq t_0$, $h > 0$ one has
        \[  
            \rho_t^\tau(x) \leq \rho_{t+h}^\tau(y) \left ( \frac{e^{\Lambda (t+h)}-1}{e^{\Lambda t}-1} \right )^{(1+\epsilon)d} \exp \left ( \frac{1}{2(h-\tau)} |x-y|^2 + \frac{h}{2} A + C (h+h^{-1}+1) \tau \right )
        \]
    \end{theorem}

    (in the case $\Lambda = 0$, the term raised to the power $(1+\epsilon)d$ should be understood as $\frac{t+h}{t}$). 

    \begin{proof}
        Let $X_t := \frac{\Lambda}{2 (1-e^{-\Lambda t})}$ for $\Lambda \neq 0$, and $X_t = \frac{1}{2 t}$ else. Also fix $\overline{\epsilon} \geq \epsilon$. By the asymptotic Li-Yau-Hamilton, we can find $\tau(\epsilon,t_0)$ depending only on $V,W$ such that $D^2 u[\rho_t^\tau] \succeq -(1+\epsilon) X_t$ for all $t \geq t_0$, $\tau \leq \tau(\epsilon,t_0)$. \\

        Let's consider $k$ be such that $t \in [k \tau, (k+1) \tau)$, and $l$ be such that $t + h \in [l \tau, (l+1) \tau)$. Let $k \leq i \leq l-1$. Up to taking $\tau(\epsilon,t_0)$ smaller, we can assume that $k \geq 2$. Let $S_i^\tau$ be the transport map from $\rho_i^\tau$ to $\rho_{i+1}^\tau$, then one has the Monge-Ampère equation
        \[
            \log \rho_i^\tau(x) + \log \det(I + \tau D^2 u[\rho_{i+1}^\tau])(S_i^\tau x) = \log \rho_{i+1}^\tau(S_i^\tau x)
        \]
        by the asymptotic estimate, we have
        \[
            \log \det(I + \tau D^2 u[\rho_{i+1}^\tau])(x) \geq d \log(1 + \tau (1+\epsilon) X_{\tau i}) 
        \]
        furthermore, one can chose a pair of Kantorovich potentials $(\phi_i^\tau,\psi_i^\tau)$ from $\rho_{i+1}^\tau$ to $\rho_i^\tau$ such that $\tau u[\rho_{i+1}^\tau] = -\phi_i^\tau$. Using that $\phi_i^\tau(y) + \psi_i^\tau(x) \leq \frac{1}{2} |x-y|^2$ with equality if $y = S_i^\tau x$ we have
        \begin{align*}
            \tau u[\rho_{i+1}^\tau](S_i x) &= -\frac{1}{2} |S_i^\tau x - x|^2 + \psi_i^\tau(x) \\
            &\leq \frac{1}{2} |x-y|^2 - -\frac{1}{2} |S_i^\tau x - x|^2 - \phi_i^\tau(y)
            \\
            &= \frac{1}{2} |x-y|^2 - -\frac{1}{2} |S_i^\tau x - x|^2 + \tau u[\rho_{i+1}^\tau](y)
        \end{align*}
        this gives
        \begin{align*}
            &\log \rho_{i+1}^\tau(S_i^\tau x) = u[\rho_{i+1}^\tau](S_i^\tau x) - V(S_i^\tau x) - W * \rho_{i+1}^\tau(S_i^\tau x) \\
            &\leq \frac{1}{2 \tau} |x-y|^2 -\frac{1}{2 \tau} |S_i^\tau x - x|^2 + u[\rho_{i+1}^\tau](y) - V(S_i^\tau x) - W * \rho_{i+1}^\tau(S_i^\tau x) \\
            &= \log \rho_{i+1}^\tau(y) + \frac{1}{2 \tau} |x-y|^2 -\frac{1}{2 \tau} |S_i^\tau x - x|^2 + R_i^\tau(x,y)
        \end{align*}
        with
        \begin{align*}
             R_i^\tau(x,y) &= V(y) - V(S_i^\tau x) + W * \rho_{i+1}^\tau(y) - W * \rho_{i+1}^\tau(S_i^\tau x) \\
             &\leq ([V]_{\Lip} + [W]_{\Lip}) |y-S_i^\tau x| \\
             &\leq ([V]_{\Lip} + [W]_{\Lip}) |x-y| + ([V]_{\Lip} + [W]_{\Lip}) |x-S_i^\tau x| \\
             &\leq ([V]_{\Lip} + [W]_{\Lip}) |x-y| + \frac{1}{2 \tau} |S_i^\tau x - x|^2 + \frac{\tau}{2}  ([V]_{\Lip} + [W]_{\Lip})^2
        \end{align*}
        which finally gives
        \[
            \log \rho_i^\tau(x) + d \log(1 + \tau(1+\epsilon) X_{\tau (i+1)}) \leq \log \rho_i^\tau(y) + \frac{1}{2 \tau} |x-y|^2 + A |x-y| + \frac{A^2}{2} \tau 
        \]
        with $A :=  [V]_{\Lip} + [W]_{\Lip}$
        which is valid for any $(x,y) \in \bb{T}^d$. 
        
        We shall now sum these estimates at the points $(x_i,x_{i+1})$ where $x_i = x + \frac{i-k}{l-k}(y-x)$ for $i=k,\ldots,l$, so that $|x_{i+1} - x_i| = \frac{1}{l-k} |x-y|$. We obtain
        \begin{align*}
            &\log \rho_k^\tau(x) + \sum_{i=k}^{l-1} d \log (1 + \tau(1+\epsilon) X_{\tau(i+1)}) \\
            &\leq \log \rho_l^\tau(y) + \frac{1}{2 \tau (l-k)} |x-y|^2 + \frac{A \tau}{\tau (l-k)} |x-y| + \frac{A^2}{2} (k-l) \tau
        \end{align*}
        As $h - \tau \leq \tau (k-l) \leq h + \tau$ this gives the bound
        \begin{align*}
            &\log \rho_t^\tau(x) + \sum_{i=k}^{l-1} d \log (1 + \tau(1+\epsilon) X_{\tau(i+1)}) \\
            &\leq \log \rho_{t+h}^\tau(y) + \frac{1}{2 (h-\tau)} |x-y|^2 + \frac{A \tau}{h-\tau} |x-y| + \frac{A^2}{2} (h+\tau)
        \end{align*}

        It remains to estimate the last sum. We shall use the inequality, valid for $x \in [0,1)$, $\log(1-x) \geq -x - \frac{x^2}{(1-x)^2}$. We also notice that $X_s$ is bounded by some $m_0$ (depending only on $t_0,V,W$) uniformly on $[t_0,+\infty)$, so that we have 
        \[
            \frac{\tau^2 (1+\epsilon)^2 X_{\tau i}^2}{(1-\tau (1+\epsilon)^2 X_{\tau i})^2 } \leq \tau^2 C(t_0,\overline{\epsilon},V,W)
        \]
        for all $\tau < \tau(t_0,\epsilon)$. Hence we obtain
        \begin{align*}
            \sum_{i=k}^{l-1} d \log(1+ \tau(1+\epsilon) X_{\tau(i+1)}) &\geq -d \tau (1+\epsilon) \sum_{i=k+1}^l X_{\tau i} - \tau^2 (l-k) C \\
            &\geq - d \tau (1+\epsilon) \sum_{i=k+1}^l X_{\tau i} - \tau (h + \tau) C
        \end{align*}
        On the other hand, there exists $M(t_0,V,W)$ such that $X_s$ is $M$-Lipschitz on $[t_0,+\infty)$. This gives
        \begin{align*}
            -\tau \sum_{i=k+1}^l X_{\tau i} &= -\sum_{i=k+1}^l \int_{\tau(i-1)}^{\tau i} X_{\tau i} \dd{s} \geq -\sum_{i=k+1}^l \int_{\tau(i-1)}^{\tau i} X_s + M \tau  \dd{s} \\
            &= -\int_{\tau k}^{\tau l} X_s \dd{s} - M \tau^2 (l-k) \geq -\int_t^{t+h} X_s \dd{s} - m_0 \tau - M \tau (h + \tau)
        \end{align*}
        Combining these estimates, we get, for another constant $C(t_0,\overline{\epsilon},V,W)$ that
        \begin{align*}
            \sum_{i=k}^{l-1} d \log(1+ \tau(1+\epsilon) X_{\tau(i+1)}) &\geq -d \int_t^{t+h} X_s \dd{s} - C \tau (h + 1) \\
            &= -\frac{d}{2} \log \frac{e^{\Lambda (t+h)}-1}{e^{\Lambda t}-1} - C \tau (h+1)
        \end{align*}
        Plugging this into the previous inequality gives the final result, eventually for another constant. 
    \end{proof}
   
\subsection{$L^2_{t,\rm{loc}} H^2_x((0,T] \times \bb{T}^d)$-convergence}
    In the case of no potential of interaction, the uniform lower bound on the Hessian for positive time allows to show that the convergence of the flow is actually stronger that a weak convergence. The proof is based the following strong convergence in the case of regular initial data, proved by Santambrogio and Toshpulatov in \cite{L2H2}

    \begin{theorem}[$L^2_t H^2_x$-convergence for Fokker-Planck equation \cite{L2H2}]
        Let $\Omega$ be a uniformly convex bounded domain of $\bb{R}^d$, $V$ of class $C^2(\overline{\Omega})$ and $\rho_0$ of class $W^{1,p}(\Omega)$, $p > d$, bounded away from $0$ and $+\infty$. Let $(\rho_t^\tau)_{t \geq 0}$ be the piecewise constant interpolation for the JKO scheme associated to the Fokker-Planck equation. Then $\rho^\tau \to \rho$ strongly in $L^2 H^2([0,T] \times \Omega)$, unique solution to the Fokker-Planck equation starting from $\rho_0$.  
    \end{theorem}

    A careful inspection of the proof shows that one can obtain the same convergence in \\ $L^2([t_0,T] ; H^2(\bb{T}^d))$, for $t_0 > 0$, still in the case $W = 0$, provided that :
    \begin{itemize}
        \item $\scr{F}[\rho_0] < +\infty$ (in order to have convergence of the JKO scheme). 
        \item $\rho_{t_0}^\tau$ satisfies is bounded in Lipschitz norm, and bounded away from $0$ and $+\infty$, uniformly in $\tau \ll 1$. 
        \item The Fisher's information $\int |\nabla \log \rho_{t_0}^\tau + V|^2 \diff \rho_{t_0}^\tau$ converges to the Fisher's information at time $t_0$ of the solution to the Fokker's Planck equation starting from $\rho_0$. 
    \end{itemize}

    Hence providing these points will give us the following strong local in time convergence of the flow:

    \begin{theorem}[$L^2_{t,loc} H^2_x((0,T) \times \bb{T}^d)$ convergence]
        Suppose $\rho_0$ is such that $\scr{F}[\rho_0] < +\infty$, let $(\rho_t^\tau)_{t \geq 0}$ piecewise constant interpolation for the JKO flow starting from $\rho_0$. Let $t_0 > 0$, then if $(\rho_t)_{t \geq 0}$ is the solution to the Fokker-Planck equation starting from $\rho_0$, one has $\rho^\tau \to \rho$ in $L^2_t H^2_x([t_0,T] \times \bb{T}^d)$
    \end{theorem}

    \begin{proof}
        The uniform Lipschitz, and away from zero and infinity bounds follows from the estimates \ref{prop: universal_bounds}. For the last point we have:
        \begin{enumerate}
            \item As $|\nabla (\log \rho_{t_0}^\tau + V)|$ is uniformly bounded, and $\rho_0^{t_0}$ is uniformly bounded away from $0$ and $+\infty$ for $\tau$ small enough, $\log \rho_{t_0}^\tau + V$ is converging uniformly in $\bb{T}^d$ up to subsequence. But since $\rho_{t_0}^\tau \to \rho_{t_0}$ a.e., we must have $\log \rho_{t_0}^\tau + V \to \log \rho_{t_0} + V$ uniformly. 
            \item We can find a constant (depending on $t_0$) such that for all $\tau$ small enough, we have $\log \rho_{t_0}^\tau + V$ is $-C$-convex. This implies that any sequence of sub-gradient for $\log \rho_{t_0}^\tau + V$ converges to a sub-gradient for $\log \rho_{t_0} + V$. Working with a countable subsequence, we obtain that $\nabla (\rho_{t_0}^{\tau_k} + V) \to \nabla (\rho_{t_0} + V)$ a.e. for any subsequence $(\tau_k)_{k \geq 0}$. 
            \item Hence we have $|\nabla (\log \rho_{t_0}^{\tau_k} + V)|^2 \rho_{t_0}^{\tau_k} \to |\nabla (\log \rho_{t_0} + V)|^2 \rho_{t_0}$ a.e. along any subsequence. Since those terms are bounded uniformly in $\tau$, we can apply the dominated convergence to deduce that
            \[
                \int |\nabla (\log \rho_{t_0}^{\tau_k} + V)|^2 \diff \rho_{t_0}^{\tau_k} \to  \int |\nabla (\log \rho_{t_0} + V)|^2 \diff \rho_{t_0}
            \]
            and the convergence then holds along
            $\tau \to 0$. 
        \end{enumerate}
    \end{proof}

\bibliography{reference}{}
\bibliographystyle{plain}
    
\appendix

\section{The One-Step estimate in the non-regular case}

    The goal of this section is to provide a proof of the one-step estimate in the non-regular case. In the case $\lambda_0 = +\infty$ there is nothing to prove, as the Kantorovich potential is $1$-concave, we have $D^2 \log \rho \succeq -1/\tau$. For the case $\lambda_0 < +\infty$, we start with the following observation : in the proof in the regular case, all computations after obtaining the inequality \ref{eq: Optimality_Equation} in the \emph{Maximum Principle} step can be done only assuming that $\eta$ satisfies $D^2 u[\eta] > -\lambda_0 > -\infty$, as in this case we have $\rho,\psi,\phi$ are of class $C^{2,\alpha}$ for all $\alpha < 1$, and all the computations only uses at most second order quantities. 

Hence to prove the result, it is sufficient to show the following:

\begin{proposition} \label{prop: sufficient_prop}
    Suppose $D^2 u[\eta] \geq -\lambda_0 > -\infty$. Let $\rho \in \Prox_\tau[\eta]$. Then keeping the notation of the proof of theorem \ref{thm: One-Step_Improvement}, there exists a maximum $(x,\nu)$ of $D^2 v(x)[\nu,\nu]$, such that denoting by $M$ the value of this maximum, one has at $x$
    \begin{equation}
        -\lambda_0 \leq M^2 u[\rho]_{\nu \nu}(\nabla v) + V_{\nu \nu} - M^2 V_{\nu \nu}(\nabla v) + W_{\nu \nu} * \eta - M^2 W_{\nu \nu} * \rho(\nabla v)
    \end{equation}
\end{proposition}

We shall rely on a combination of finite-differences to circumvent the lack of regularity, combined with a limiting argument, to obtain that one can find an inequality at a maximum point for the eigenvalues of the Hessian involving only quantities of order $1$ and $2$ without using two derivatives. 

The use of finite-differences approximation to use maximum principle when one cannot use it directly (either because of a lack of regularity, or because the maximum does not exist) can be tracked back to Caffarelli in his proof of his famous Caffarelli's contraction theorem \cite{CaffContract}. It has be then widely used in similar contexts: for instance, in \cite{MomentMeasure} for moment measures, and in \cite{CarlierSantamFig} to prove a modified version of the contraction theorem. 

\subsection{Finite Differences and Approximation} 

    For a function $f$, possibly vector valued, and for $\nu \in \bb{S}^d$, $h > 0$ we define the first and second order finite differences
    \begin{align}
        &\delta_{h,\nu} f(x) := f(x + h \nu) - f(x - h \nu) \\
        &\Delta_{h,\nu} f(x) = f(x + h \nu) + f(x-h \nu) - 2 f(x)
    \end{align}
    those are approximation of the classical first and second order derivative of $f$ in direction $\nu$. More precisely, we have the following
    
    \begin{lemma} \label{lemma: finite_diff_approx}
        Suppose $f$ is of class $C^{2,\alpha}(\bb{T}^d)$, then 
        \begin{align*}
            &\frac{1}{2 h} \delta_{h,\nu} f(x) \to \nabla f_\nu(x) \\
            &\frac{1}{h^2} \Delta_{h,\nu} f(x) \to D^2 f(x)[\nu,\nu] 
        \end{align*}
        uniformly in $(x,\nu) \in \bb{T}^d \times \bb{S}^d$. \\
        Furthermore, if $(x_h,\nu_h)$ is a minimizer of $\Delta_{h,\nu} f(x)$ for $h > 0$ (which exists by continuity), then one can find a subsequence $(x_k,\nu_k,h_k)$, with $h_k \to 0$ converging to $(x_0,\nu_0)$ minimizers of $D^2 f(x)[\nu,\nu]$ as $k \to +\infty$. 
    \end{lemma}
    
    \begin{proof}
        Fix $h > 0$, $x \in \bb{T}^d$ and $\nu \in \bb{S}^d$. Since $u$ is of class $C^{2,\alpha}$, one can find $\xi,\eta \in [-h,h]$ such that $(2 h)^{-1} \delta_{h,\nu} f(x) = \nabla f_\nu(x + \xi \nu) = D^2 f(x + \xi \nu) \cdot \nu$ and $h^{-2} \Delta_{h,\nu} f(x) = D^2 f(x + \eta \nu)[\nu,\nu]$. But then we can bound
        \begin{align*}
            &|(2 h)^{-1} \delta_{h,\nu} f(x) - \nabla f_\nu(x)| \leq ||D^2 f(x + \xi \nu) - D^2 f(x)|| \leq C |\xi|^\alpha \leq C h^\alpha \\
            &|h^{-2} \Delta_{h,\nu} f(x) - D^2 f(x)[\nu,\nu]| \leq ||D^2 f(x + \eta \nu) - D^2 f(x)|| \leq C |\eta|^\alpha \leq C h^\alpha
        \end{align*}
        which concludes the uniform convergence. 
    
        Now using this uniform convergence, it is easy to see that $h^{-2} \min_{x,\nu} \Delta_{h,\nu} f(x) \to \min_{x,\nu} D^2 f(x)[\nu,\nu]$. Since $\bb{T}^d \times \bb{S}^d$ is compact, to any sequence of minimizers $(x_h,\nu_h)$ we can extract a converging subsequence along $h_k \to 0$, and passing to the limit in the above inequality we deduce that the limit point if a minimizer. 
    \end{proof}

\subsection{An Inequality at a Maximum Point}

    From now on, we fix a probability measure $\eta$ on the torus, with $D^2 u[\eta] \geq -\lambda_0$ in the weak sense. This implies in particular that $\eta \in C^{0,1}(\bb{T}^d)$ and $\eta > 0$, and we let $\rho \in \Prox_\tau[\eta]$. We also consider $(\psi,\phi)$ pair of Kantorovich potential from $\rho$ to $\eta$. Since $\rho > 0$ and $\tau u[\rho] = -\psi$ is of class $C^{0,1}(\bb{T}^d)$, we get that $\rho$ is of class $C^{0,1}(\bb{T}^d)$ positive, hence the Kantorovich potentials are of class $C^{2,\alpha}(\bb{T}^d)$ for all $\alpha < 1$ and satisfies $D^2 \phi, D^2 \psi \preceq 1-\epsilon$ for some $\epsilon > 0$.

    We let $v := |x|^2 / 2 - \phi$. By the above considerations, we see that the Monge-Ampere equation holds everywhere, that is,
    \[
        \log \det D^2 v + \log \rho(\nabla v) = \log \eta
    \]

    \begin{proposition} \label{prop: ineq_finite_diff_opt}
        Let $(x_h,\nu_h)$ minimizing $\Delta_{h,\nu} \phi(x)$, there exists $\alpha_h \leq 0$ such that $\delta_{h,\nu_h} \nabla \phi(x_h) = (h-\alpha_h )\nu_h$, and if $y_h = \nabla v(x_h)$, one has
        \begin{equation} \label{eq: ineq_finite_diff}
            -h^2 \lambda_0 \leq \Delta_{\alpha_h,\nu_h} u[\rho](y_h) + \Delta_{h,\nu_h} V(x_h) - \Delta_{\alpha_h,\nu_h} V(y_h) + \Delta_{h,\nu_h} W * \eta(x_h) - \Delta_{\alpha_h,\nu_h} W * \rho(y_h)
        \end{equation}
    \end{proposition}

    \begin{proof}
        
        We proceed into several steps. We shall drop the dependency on $h$ of $x_h,\nu_h,\alpha_h$. First note that since $\Delta_{h,\nu} |\cdot|^2 = 2 h^2$ minimizing $\Delta_{h,\nu} \phi(x)$ is the same as maximizing $\Delta_{h,\nu} v(x)$. 

        \begin{itemize}
        
            \item \emph{Maximum Principle : } By optimality at $x$ we obtain
            \[
                D^2 \Delta_{h,\nu} v(x) \preceq 0 \quad \nabla v(x + h \nu) + \nabla v(x-h \nu) = 2 \nabla v(x)
            \]
            and by optimality in $\nu$ we obtain that there exists $\beta$ real such that 
            \[
                \delta_{h,\nu} \nabla v(x) = \nabla v(x + h \nu) - \nabla v(x-h \nu) = 2 \alpha \nu 
            \]
            Since $\nabla v = x - \nabla \phi$ we also obtain $\delta_{h,\nu} \nabla \phi = 2 (h-\alpha) \nu$. Note furthermore that, by convexity, $t \to v_\nu(x + t \nu) - v_\nu(x-t \nu)$ is non decreasing, hence $\alpha \geq 0$.  
            Finally, combining the first order condition in $x$ and $\nu$ we get
            \[
                \nabla v(x \pm h\nu) = \nabla v(x) \pm \alpha \nu
            \]

            \item \emph{Inequality by Monge-Ampere equation : } We rewrite the Monge-Ampere equation as
            \[
                J[D^2 v] + u[\rho](\nabla v) + V - V(\nabla v) + W * \eta - W * \rho(\nabla v) = u[\eta]
            \]
            with $J[B] = \log \det B$. By concavity of the log-determinant we observe that
            \[
                \Delta_{h,\nu} J[D^2 v] \leq \Tr [D^2 v]^{-1} D^2 \Delta_{h,\nu} v 
            \]
            in particular, by second order optimality condition this term is negative at $x$. On the other hand, since $u[\eta]$ is $-\lambda_0$-convex, one can derive 
            \[
                \Delta_{h,\nu} u[\eta] \geq -h^2 \lambda_0
            \]
            Hence we obtain
            \[
                -h^2 \lambda_0 \leq \Delta_{h,\nu} u[\rho] \circ \nabla v + R
            \]
            with $R_h$ being the operator $\Delta_{h,\nu}$ applied to the remaining functions. More precisely
            \[
                R = \Delta_{h,\nu} V - \Delta_{h,\nu} V \circ \nabla v + \Delta_{h,\nu} W * \eta - \Delta_{h,\nu} W * \rho \circ \nabla v
            \]

            \item \emph{Rewriting the Rest : } Using that $\nabla v(x \pm h \nu) = \nabla v(x) \pm \alpha \nu$, if we let $y := \nabla v(x)$ we get
            \[
                R = \Delta_{h,\nu} V(x) - \Delta_{\alpha,\nu} V(y) + \Delta_{h,\nu} W * \eta(x) - \Delta_{\alpha,\nu} W * \rho(y) 
            \]
            which is exactly the remainder that we have above. 

            \item \emph{The $u[\rho]$ term : } Again, using that $\nabla v(x \pm h \nu) = \nabla v(x) \pm \alpha \nu$ we obtain $\Delta_{h,\nu} u[\rho] \circ \nabla v = \Delta_{\alpha,\nu} u[\rho](x)$. 
        \end{itemize}
        Combining all of this gives the inequality.
    \end{proof}

\subsection{Proof of proposition \ref{prop: sufficient_prop}} 

    Using the previous proposition, we will obtain proposition \ref{prop: sufficient_prop} by passing to the limit along a converging subsequence. 

    \begin{proof}[Proof of proposition \ref{prop: sufficient_prop}]
        Let's consider $(x_h,y_h,\nu_h,\alpha_h)$ as given by proposition \ref{prop: ineq_finite_diff_opt}. Since $\phi$ is of class $C^{2,\alpha}$, we can take a subsequence such that $x_k,\nu_k$ converging to a point $(x,\nu)$ minimizing $D^2 \phi(x)[\nu,\nu]$. Or equivalently, $(x,\nu)$ maximizes $D^2 v(x)[\nu,\nu]$. From now on, all limit are taken along this particular subsequence. 

        We then have the following convergences
        \begin{itemize}
            \item By continuity of $\nabla v$, $y_h = \nabla v(x_h) \to \nabla v(x) =: y$. 
            \item By $C^{2,\alpha}$-continuity of $\phi$ we have $h^{-1} \delta_{h,\nu_h} \nabla \phi(x_h) \to \nabla \phi_\nu(x)$. 
            \item As $h^{-1} \delta_{h,\nu_h} \nabla \phi(x_h) = (1 - h^{-1} \alpha_h) \nu_h$ we must have $\nabla \phi_\nu(x) = \phi_{\nu \nu}(x) \nu$, and $h^{-1} \alpha_h \to 1 - \phi_{\nu \nu}(x) = v_{\nu \nu}(x) = M$. Furthermore $\alpha_h \to 0$. 
        \end{itemize}

        We then divides the inequality \ref{eq: ineq_finite_diff} by $h^2$ which gives
        \begin{align*}
            -\lambda_0 &\leq h^{-2} \Delta_{\alpha_h,\nu_h} u[\rho](y_h) \\
            &+ h^{-2} \Delta_{h,\nu_h} V(x_h) - h^{-2} \Delta_{\alpha_h,\nu_h} V(y_h) + h^{-2} \Delta_{h,\nu_h} W * \eta(x_h) - h^{-2} \Delta_{\alpha_h} W * \rho(y_h)
        \end{align*}

        First, as $V,W$ are $C^{2,1}(\bb{T}^d)$, we have, by lemma \ref{lemma: finite_diff_approx}, 
        \[ 
            h^{-2} \Delta_{h,\nu_h} V(x_h) \to V_{\nu \nu}(x) \qquad h^{-2} \Delta_{h,\nu_h} W * \eta(x_h) \to W_{\nu \nu} * \eta(x)
        \]
        On the other hand, since $u[\rho]$ is $C^{2,\alpha}$ one has 
        \[ 
            h^{-2} \Delta_{\alpha_h,\nu_h} u[\rho](y_h) = \frac{\alpha_h^2}{h^2} \frac{1}{\alpha_h^2} \Delta_{\alpha_h,\nu_h} u[\rho](y_h) \to M^2 u[\rho]_{\nu \nu}(y)
        \]
        similarly 
        \[
            h^{-2} \Delta_{\alpha_h,\nu_h} V(y_h) \to M^2 V_{\nu \nu}(y) \qquad h^{-2} \Delta_{\alpha_h,\nu_h} W * \rho(y_h) \to M^2 W_{\nu \nu} * \rho(y)
        \]
        Hence we obtain at the point $x$, in direction $\nu$, maximum of $D^2 v(x)[\nu,\nu]$, that
        \[
            \lambda_0 \leq M^2 u[\rho]_{\nu \nu}(y) + V_{\nu \nu}(x) - M^2 V_{\nu \nu}(y) + W_{\nu \nu} * \eta(x) - M^2 W_{\nu \nu} * \rho(y)
        \]
        which concludes since $y = \nabla v(x)$. 
    \end{proof}

\section{Asymptotic for Comparison Sequence}

    Here we give the proofs of proposition \ref{prop: Asymptotic_Estimate_Heat} and \ref{prop: Asymptotic_Estimate_Fokker_Planck_Granular_Medium}. We shall first give a proof of the first one \ref{prop: Asymptotic_Estimate_Heat}, then we shall show some preliminary results about the comparison sequence which will be used in the final part, where we prove the second proposition \ref{prop: Asymptotic_Estimate_Fokker_Planck_Granular_Medium}. Let's recall the value of $G$ for the sake of readability:
\[
    G[E,\tau] = \frac{E}{(1-E)^2}(1-\tau (2 \lambda^* + L^*) + \tau (\lambda^* + L^*) E)
\]

\subsection{Proof of proposition \ref{prop: Asymptotic_Estimate_Heat}: The Heat Case}

    Let us first consider the case $\lambda_0 = +\infty$. Note that in the Heat case one has 
    \[
        G[E,\tau] = G[E] = \frac{E}{(1-E)^2}
    \]
    
    \begin{lemma} \label{lemma: Asymptotic_Estimate_Heat_Infinite}
        If $\lambda_0 = +\infty$, then $E_k^\tau$ does not depends on $\tau$, and one has $E_k \sim \frac{1}{2k}$ as $k \to +\infty$.
    \end{lemma}

    \begin{proof}
        The non-dependency on $\tau$ is trivial as the initial data is $+\infty$ and $G$ does not depend on $\tau$. By a fixed point and monotonicity argument, we easily see that $E_k \to 0$ as $k \to +\infty$. On the other hand we observe that
        \[
            \frac{1}{E_{k+1}} - \frac{1}{E_k} = \frac{1 - (1-E_{k+1})^2}{E_{k+1}} = 2 - E_{k+1} \underset{{k \to +\infty}}{\to} 2
        \]
        using Cesaro's lemma we obtain that
        \[
            \frac{1}{E_k} - \frac{1}{E_0} = \sum_{i=0}^{k-1} \frac{1}{E_{k+1}} - \frac{1}{E_k} \sim 2 k
        \]
        which shows that $E_k \sim \frac{1}{2k}$. 
    \end{proof}

    For the $\lambda_0 < +\infty$ case, we recall that we want to prove:

    \begin{lemma} \label{lemma: Asymptotic_Estimate_Heat_Finite}
        Suppose $\tau \lambda_0 \leq 1$, then for all $k \geq 0$ we have
        \[
            \frac{1}{\tau} E_k^\tau \leq \frac{\lambda_0}{\tau k \lambda_0 (2 - \tau \lambda_0) + 1}
        \]
    \end{lemma}
    
    we shall make use of the following functions: we define for $z \geq 1$, $f(z) := z - \sqrt{z(z-1)}$, and $g(z) = z^{-1} f(z)$. One observe that $f$ is decreasing, valued in $(1/2,1]$, and that $g$ is a decreasing diffeomorphism from $[1,+\infty)$ to $(0,1]$ with inverse $g^{-1}(z) = \frac{1}{z(2-z)}$. We shall rely on the following lemma:

    \begin{proof}
        Let $K \geq 1$, then for all $k \geq 0$ we have
        \begin{equation}
            E_k^\tau \leq \frac{\max(\tau K \lambda_0, f(K)}{k + K}
        \end{equation}
    \end{proof}

    \begin{proof}
        Consider a step horizon $k_0 \geq 1$, and a time step $k^* \in \{0,\ldots,k_0\}$ such that $(k+K) E_k^\tau$ is maximal. We distinguish two cases:
        \begin{itemize}
            \item If $k^* = 0$, then we get that $(k+K) E_k^\tau \leq \tau K \lambda_0$ for all $0 \leq k \leq k_0$.
            \item Else, we have $k^* \geq 1$, let $M$ be the value of the maximum, so that $(k^* + K) E_{k^*}^\tau = M$ and $(k^*-1+K) E_{(k^*-1)}^\tau \leq M$. Using the inductive definition of $(E_k^\tau)_{k \geq 0}$ this gives
            \[
                (k^* - 1 + K) \frac{E_{k^*}}{(1-E_{k^*})^2} \leq M
            \]
            replacing $E_{k^*}$ by $M / (k^* + K)$ and after algebraic manipulation we obtain
            \[
                (k^* + K - M)^2 M \geq (k^*+K)(k^*+K-1) M
            \]
            either $M = 0$ and there is nothing to prove, or using that $E_{k^*}^\tau \leq 1$, implying that $M \leq k^* + K$ to get rid of the square root, we obtain $M \leq f(k^* + K)$, hence $(k+K) E_k^\tau \leq f(K)$ for all $k \leq k_0$. 
        \end{itemize}
        Now observe that if $f(K) < \tau K \lambda_0$, we must be on the first case, since else we would also have $\tau \lambda_0 K \leq f(K)$. Thus we deduce that in both cases we have
        \[
            E_k^\tau \leq \frac{\max(\tau K \lambda_0, f(K)}{k+K}
        \]
        for all $k \leq k_0$, letting $k_0 \to +\infty$ gives the final result. 
    \end{proof}

    \begin{proof}[Proof of lemma \ref{lemma: Asymptotic_Estimate_Heat_Finite}]
        We want to chose $K$ such attaining the maximum, i.e. such that $g(K) = \tau \lambda_0$, this gives, under the assumptions that $\tau \lambda_0 \leq 1$, $K = g^{-1}(\tau \lambda_0) = \frac{1}{\tau \lambda_0 (2 - \tau \lambda_0)}$. Plugging this value into the previous estimate gives
        \[
            E_k^\tau \leq \frac{\tau \lambda_0}{\tau k \lambda_0 (2 - \tau \lambda_0) + 1}
        \]
        which is what we wanted to show. 
    \end{proof}

\subsection{Preliminary results for the Proof of the Granular-Medium case}

    We shall now prove some preliminary results for the Granular-Medium and Fokker-Planck case. From now on, we assume that at least one of the potential is non-zero. In particular, this implies that $\lambda^*,L^* > 0$. We shall also always assume $\tau < \tau^*$ so that the comparison sequence is well-defined. We also fix $\lambda_0 \in [0,+\infty]$. 

    The first result is a give the behaviour, at $\tau$ fixed, of the comparison sequence as $k \to +\infty$.

    \begin{lemma}
        Define the critical value
        \begin{equation} \label{eq: Critical_Value}
            E_c^\tau = \left ( 1 + \tau \frac{\lambda^* + L^*}{2} \right ) \left ( 1 - \sqrt{1 - 4 \tau \frac{2 \lambda^* + L^*}{(2 + \tau (\lambda^* + L^*))^2}} \right )
        \end{equation}
        for all $\tau < \tau^{**} =: \min(\tau^*, \frac{1}{\lambda^*})$, so that $E_c^\tau$ is a well-defined element of $[0,1)$. 
        Then:
        \begin{itemize}
            \item If $\tau \lambda_0 \in \{0,E_c^\tau \}$, then $E_k^\tau$ is constant.
            \item If $\tau \lambda_0 \in (0,E_c^\tau)$, then $E_k^\tau$ is increasing converging to $E_c^\tau$.
            \item If $\tau \lambda_0 > E_c^\tau$, then $E_k^\tau$ is decreasing converging to $E_c^\tau$. 
        \end{itemize}
    \end{lemma}

    \begin{proof}
        One compute that
        \[
            G[E,\tau] - E = -\frac{E}{(1-E)^2}(E^2 - (2 + \tau(\lambda^* + L^*)) E + \tau (2 \lambda^* + L^*))
        \]
        which has the reverse sign as the polynomial $P(E) = E^2 - (2 + \tau (\lambda^* + L^*)) E + \tau (2 \lambda^* + L^*)$. Since $P(0) = \tau (2 \lambda^* + L^*) > 0$, and $P(1) = \tau \lambda^* - 1 < 0$ if $\tau < \frac{1}{\lambda^*}$, we have at least one root in $[0,1)$ for $\tau$ small enough, and as the sum of the root is $2 + \tau (\lambda^* + L^*) > 1$, only the smallest one is in this set. This root is then exactly
        \[
            E_c^\tau = \left ( 1 + \tau \frac{\lambda^* + L^*}{2} \right ) \left ( 1 - \sqrt{1 - 4 \tau \frac{2 \lambda^* + L^*}{(2 + \tau (\lambda^* + L^*))^2}} \right )
        \]
        This shows that $G[E,\tau] < E$ on $(0,E_c^\tau)$, $G[E,\tau] > E$ on $(E_c^\tau,1]$ and $G[E,\tau] = E$ on $\{0,E_c^\tau\}$, and we conclude about the monotonicity and limit of $E_k^\tau$ by fixed point argument. 
    \end{proof}

    From now on we will always assume that $\tau < \tau^{**}$. 

    \begin{remark}
        As $\tau \to 0$, one has $E_c^\tau \sim \tau \frac{2 \lambda^* + L^*}{2} = \tau \frac{\Lambda}{2}$, which is the correct behavior one would expect by looking at the continuous case. 
    \end{remark}

    This result shows that $E_k^\tau$ shall be converging to 0 when $\tau \to 0$ and $k \to +\infty$, but to make this observation precise, we shall need some kind of uniform convergence in those two variables. This is the content of the following result.

    \begin{lemma}[Uniform convergence to $0$] \label{lemma: Uniform_Convergence_0}
        Let $\delta > 0$, then there exists $\tau(\delta) > 0$ and $k(\delta) \geq 0$ such that for all $\tau < \tau(\delta)$ and $k \geq k(\delta)$, we have $E_k^\tau \leq \delta$. 
    \end{lemma}

    In order to prove this result, we first need the following lemma.

    \begin{lemma} \label{lemma: Generalized_Comparison}
        Suppose $\tau \leq \eta$, then for all $k \geq 0$, one has $E_k^\tau \leq E_k^\eta$. 
    \end{lemma}

    \begin{proof}
        Let us write $G[E,\tau] = G[E] + \tau R[E]$ with $G[E] = \frac{E}{(1-E)^2}$ and $R[E] = \frac{E}{(1-E)^2}((\lambda^* + L^*) E - (2 \lambda^*+L^*))$. We notice that $R[E] \leq 0$ on $[0,1)$.
        
        We argue by contradiction. We assume that we can find some $k$ such that $E_k^\eta < E_k^\tau$, and consider the minimal such $k$. As $E_0^\tau = \tau \lambda_0 \leq \eta \lambda_0 = E_0^\eta$ we must have $k \geq 1$. But then using that $G$ is increasing we have
        \[
            E_{k-1}^\eta = G[E_k^\eta,\eta] < G[E_k^\tau,\eta] = G[E_k^\tau,\tau] + (\eta-\tau) R[E_k^\tau] = E_{k-1}^\tau + (\eta-\tau) R[E_k^\tau]
        \]
        but as $\eta \geq \tau$, and $R[E_k^\tau] \leq 0$, we deduce that $E_{k-1}^\eta < E_{k-1}^\tau$, which contradicts the minimality of $k$. Hence we deduce that $E_k^\tau \leq E_k^\eta$ for all $k \geq 0$. 
    \end{proof}

    \begin{proof}[Proof of lemma \ref{lemma: Uniform_Convergence_0}]
        As $E_c^\tau \to 0$ when $\tau \to 0$, we can first find $\tau(\delta)$ such that $E_c^{\tau(\epsilon)} \leq \frac{\delta}{2}$. But then using that $E_k^{\tau(\delta)} \to E_c^{\tau(\delta)}$, we can find $k(\delta)$ such that $E_k^{\tau(\delta)} \leq E_c^{\tau(\delta)} + \frac{\delta}{2} \leq \delta$ for all $k \geq k(\delta)$. Then using lemma \ref{lemma: Generalized_Comparison} above, we have that for all $\tau < \tau(\delta)$ and $k \geq k(\delta)$ :
        \[
            E_k^\tau \leq E_k^{\tau(\delta)} \leq \delta
        \]
    \end{proof}
    
\subsection{Proof of proposition \ref{prop: Asymptotic_Estimate_Fokker_Planck_Granular_Medium}: The Granular-Medium Case}

    To prove this case, we shall need to look into another sequence, defined as $X_k^\tau := \frac{\tau}{E_k^\tau}$ (assuming $\lambda_0 \neq 0$). We shall see that this modified sequence can be compared quantitatively to the solution of some ODE as $\tau \to 0$ and $k \tau \sim t$. This estimate will then in turn be used to prove our final proposition. We introduce the following new function:
    \begin{equation}
        H[X,\tau] := \frac{(X-\tau)^2}{[1 - \tau (2 \lambda^* + L^*)] X + \tau^2 (\lambda^* + L^*)}
    \end{equation}
    
    \begin{definition}[Inverse Comparison-Sequence]
        If $\lambda_0 > 0$, we define the inverse comparison sequence by $X_k^\tau := \frac{\tau}{E_k^\tau}$. It satisfies
        \begin{equation}
            \left \{ \begin{array}{ll} 
                X_0 = \frac{1}{\lambda_0} \in [0,+\infty) & \\
                X_k^\tau \in [\tau,+\infty) & \forall k \geq 1 \\
                H[X_{k+1}^\tau,\tau] = X_k^\tau & \forall k \geq 0
            \end{array} \right .
        \end{equation}      
    \end{definition}

    The fundamental observation is that one has $H[X,0] = X$, and $\partial_\tau H[X,0] = \Lambda X - 2$. If we linearize we then expect that $X_k^\tau \simeq X_{k+1}^\tau + \tau \partial_\tau H[X_{k+1}^\tau,0]$, hence if $X_t^\tau$ is the piecewise constant interpolation of the values of $X_k^\tau$, one has have
    \[
        \dot{X}_t^\tau \simeq \frac{X_{k+1}^\tau - X_k^\tau}{\tau} \simeq -\partial_\tau H[X_t^\tau,0] = 2 - \Lambda X_t^\tau
    \]
    hence $X_t^\tau$ should not be too far from the solution to $\dot{X}_t = 2 - \Lambda X_t$, which is exactly the equation we obtained in the continuous case. 

    Our aim is thus to prove rigorously this fact. To do so we shall introduce $R[X,\tau] := H[X,\tau] - X - \tau \partial_\tau H[X,0]$, measuring the error we make in our linearization. We then have the following estimate

    \begin{lemma}[Estimating the rest]
        There exists a constant $C > 0$ such that for all $\tau \ll 1$, and $X > 0$
        \begin{equation}
            |R[X,\tau]| \leq C \tau^2 \frac{1+X^2}{X}
        \end{equation}
    \end{lemma}

    \begin{proof}
        We have the following explicit expression:
        \[
            R[X,\tau] = \tau^2 \frac{1 - 2 \Lambda X + \Lambda^2 X - \beta X + 2 \tau \beta - \tau \beta \Lambda X}{[1 - \tau (2 \lambda^* + L^*)] X + \tau^2 (\lambda^* + L^*)}
        \]
        where $\beta = \lambda^* + X^*$
        The denominator can be bounded from below by $\frac{1}{2} X$ as long as $\tau < \frac{1}{2 (2 \lambda^* + L^*)}$ for instance. Taking $\tau \leq 1$ allows to bound the numerator by $C (1 + X^2)$ for some large enough constant depending only on $\Lambda,\beta$, from which the result follows. 
    \end{proof}

    We can then use this estimate to derive the following quantitative linearization of the induction

    \begin{proposition} \label{prop: Comparison_With_ODE}
        Let $X_k^\tau$ be a modified comparison sequence starting from $\lambda_0 > 0$. Suppose that for some $k_0 \geq 0$, one has $0 < m \leq X_k^\tau \leq M$ for all $k \geq k_0$. Let $X_t$ be the solution to $\dot{X}_t = 2 - \Lambda X_t$ with initial data $X_0 = X_{k_0}^\tau$. Then for all $k \geq k_0$ we have
        \begin{equation}
            |X_{(k-k_0) \tau} - X_k^\tau| \leq \Lambda \tau \left | X_{k_0}^\tau - \frac{2}{\Lambda} \right | + \tau C \frac{1+M^2}{\Lambda m}
        \end{equation}
        for some constant $C$ depending only on $\lambda^*,L^*$. 
    \end{proposition}

    \begin{proof}
        Define $\Delta_k^\tau := |X_{(k-k_0) \tau} - X_k^\tau|$. We recall that
        \[
            X_k^\tau = H[X_{k+1}^\tau,\tau] = X_{k+1}^\tau + \tau (\Lambda X_{k+1}^\tau - 2) + R[X_{k+1}^\tau,\tau]
        \]
        On the other hand, there exists $\xi_k^\tau \in [(k-k_0) \tau, (k+1-k_0) \tau)$ such that 
        \[
            X_{(k+1-k_0) \tau} = X_{(k-k_0) \tau} + \tau \dot{X}_{\xi_k^\tau} = X_{(k-k_0) \tau} + 2 \tau - \tau \Lambda X_{\xi_k^\tau} 
        \]
        Combining these relations we obtain
        \begin{align*}
            X_k^\tau - X_{(k-k_0) \tau} = (1 + \tau \Lambda) (X_{k+1}^\tau - X_{(k+1-k_0) \tau}) + R[X_{k+1}^\tau, \tau] + \tau \Lambda [X_{(k+1-k_0) \tau} - X_{\xi_k^\tau}] 
        \end{align*}
        this gives the bound
        \begin{align*}
            (1 + \tau \Lambda) \Delta_{k+1}^\tau &\leq \Delta_k^\tau + |R[X_{k+1}^\tau,\tau]| + \tau \Lambda |X_{(k+1-k_0) \tau} - X_{\xi_k^\tau}| \\
            &\leq \Delta_k^\tau + C \tau^2 \frac{1+M^2}{m} + \Lambda \tau^2 \left | X_{k_0}^\tau - \frac{2}{\Lambda} \right |
        \end{align*}
        where we use the explicit formula $X_t = \left ( X_{k_0}^\tau - \frac{2}{\Lambda} \right ) e^{-\Lambda t} + \frac{2}{\Lambda}$ which implies that $X_t$ is $\left | X_{k_0}^\tau - \frac{2}{\Lambda} \right |$ to bound the last term. 

        One can then use the discrete Grönwall lemma \ref{lemma: Discrete_Gronwall_Lemma} below, we deduce that
        \[
            \Delta_k^\tau \leq \left ( \Lambda \tau \left | X_{k_0}^\tau - \frac{2}{\Lambda} \right | + \tau C \frac{1+M^2}{m \Lambda} \right ) \left ( 1 - \left ( \frac{1}{1 + \Lambda \tau} \right )^{k-k_0} \right ) 
        \]
        if we bound crudely the last term by $1$ we obtain the result. 
    \end{proof}

    \begin{lemma}[Discrete Grönwall Lemma] \label{lemma: Discrete_Gronwall_Lemma}
        Let $u_k$ be a sequence such that for some constant $\alpha > 0, A \geq 0$ one has
        \begin{equation}
            (1 + \alpha) u_{k+1} \leq u_k + A
        \end{equation}
        and $u_0 = 0$, then 
        \begin{equation}
            u_k \leq \frac{A}{\tau} \left ( 1 - \left ( \frac{1}{1+\alpha} \right )^k \right )
        \end{equation}
    \end{lemma}

    \begin{proof}
        This is classical. Consider $v_k := u_k - \frac{A}{\alpha}$, then $(1+\alpha) v_{k+1} \leq v_k$, so that $v_k \leq \frac{1}{(1+\alpha)^k} v_0 = -\frac{1}{(1+\alpha)^k} \frac{A}{\alpha}$. Replacing $v_k$ by its value gives the result. 
    \end{proof}

    We can finally prove proposition \ref{prop: Asymptotic_Estimate_Fokker_Planck_Granular_Medium}. 

    \begin{proof}[Proof of proposition \ref{prop: Asymptotic_Estimate_Fokker_Planck_Granular_Medium}]

        We consider the cases $\lambda_0 < +\infty$ and $\lambda_0 = +\infty$ separately. 

        In all the proof, $C$ will denote a constant depending only on $\lambda_0,\lambda^*,L^*$, which might change from line to line. 

        \begin{itemize} 
        
            \item \emph{The case $\lambda_0 < +\infty$: } If $\lambda_0 = 0$ there is nothing to prove, else we can work with the inverse comparison sequence starting from $\lambda_0$. Since $\lambda_0 < +\infty$, the sequence $E_k^\tau$ remains bounded between $\tau \lambda_0$ and $E_c^\tau$. As $\frac{1}{\tau} E_c^\tau \to \frac{\Lambda}{2}$, we deduce that $(X_k^\tau)_k$ remains bounded away from $0$ and $+\infty$ uniformly on $\tau$. Consider $X_t$ solution to $\dot{X}_t = 2 - \Lambda X_t$ with initial data $X_0 = \frac{1}{\lambda_0}$. Then by proposition \ref{prop: Comparison_With_ODE}, using the uniform bounds from above and below for $(X_k^\tau)_k$, then for all $k \geq 0$:
            \[
                X_k^\tau \geq X_{k \tau} - \tau C 
            \]
            
            Now let $t \geq 0$, and consider $k \geq 0$ such that $t \in [k \tau, (k+1) \tau)$, using that $X_t$ is Lipschitz, we have $X_k^\tau \geq X_t - \tau C$. Let $m$ be such that $X_t \geq m > 0$ for all $t \geq 0$ (which exists as $\lambda_0 > 0$). Fix $\epsilon > 0$, then 
            \[
                \frac{X_t- \tau C}{X_t} = 1 - \frac{\tau C}{X_t} \geq 1 - \frac{\tau C}{m} \geq \frac{1}{1+\epsilon}
            \]
            for all $\tau < \tau(\epsilon)$ small enough. We deduce that
            \[
                \frac{1}{\tau} E_t^\tau = \frac{1}{\tau} E_k^\tau \leq \frac{1+\epsilon}{X_t} = (1+\epsilon) \frac{\Lambda \lambda_0}{\Lambda e^{-\Lambda t} + 2 \lambda_0 (1-e^{-\Lambda t})}
            \]
            for all $t \geq 0$ and $\tau < \tau(\epsilon)$, which concludes. 

            \item \emph{The case $\lambda_0 = +\infty$: } In this case, we have the inequality $X_k^\tau \leq \frac{\tau}{E_c^\tau}$, which gives an uniform upper bound on $X_k^\tau$. The lower bound is a bit trickier since we have $X_0^\tau = 0$. 
            
            To solve this issue, let's consider $\delta > 0$ to be chosen later on. By lemma \ref{lemma: Uniform_Convergence_0}, we can find $\tau(\delta) > 0$ and $k(\delta) \geq 0$ such that $X_k^\tau \geq \frac{\tau}{\delta}$ for all $k \geq k(\delta)$ and $\tau < \tau(\delta)$. Let's consider $X_t$ the solution to $\dot{X}_t = 2 - \Lambda X_t$ starting from $X_{k(\delta)}^\tau$. Then using proposition \ref{prop: Comparison_With_ODE} (eventually modifying $\tau(\delta)$ for it to be smaller) with some $M$ given by an upper bound on $X_c^\tau$, and $m = \frac{\tau}{\delta}$, then for all $k \geq k(\delta)$ 
            \[
                |X_{(k-k(\delta)) \tau} - X_k^\tau| \leq \Lambda \tau \left ( |X_{k_0}^\tau| + \frac{2}{\Lambda} \right ) + \tau C \frac{1+M^2}{m} \leq C (\tau + \delta)
            \]
            Furthermore, let's consider $t \in [k \tau, (k+1) \tau)$, then 
            \[
                X_{(k-k(\delta)) \tau} = (X_{k(\delta)}^\tau - \frac{2}{\Lambda}) e^{-\Lambda (k-k(\delta) \tau} + \frac{2}{\Lambda} \geq \frac{2}{\Lambda}(1 - e^{-\Lambda t}) - \Lambda k(\delta) \tau
            \]
            Hence we end up with
            \[
                X_k^\tau \geq \frac{2}{\Lambda}(1-e^{-\Lambda t}) - C (\tau + \delta + \tau k(\delta))
            \]
            for all $k \geq k(\delta), \tau \leq \tau(\delta)$, and $t \in [k \tau, (k+1) \tau)$. 

            Now consider $\epsilon > 0$ and $t_0 > 0$. As $\frac{2}{\Lambda}(1-e^{-\Lambda t})$ is bounded uniformly from below on $[t_0,+\infty)$, we can find $\eta > 0$, depending on $t_0,\lambda^*,L^*,\epsilon$ such that $\frac{1}{1+\epsilon} \frac{2}{\Lambda}(1-e^{-\Lambda t}) \leq \frac{2}{\Lambda}(1-e^{-\Lambda t}) - 3 \eta$, we shall first chose $\delta$ such that $C \delta \leq \eta$, then we chose $\tau$ small enough such that $C k(\delta) \tau \leq \eta$,  $C \tau \leq \eta$, $\tau < \tau(\delta)$ and $\tau k(\delta) \leq t_0$. Then if $t \in [t_0,+\infty)$, and $k$ is such that $t \in [k \tau, (k+1) \tau)$, we must have $k \geq k(\delta)$, therefore we obtain
            \begin{align*}
                X_k^\tau &\geq \frac{2}{\Lambda}(1-e^{-\Lambda t}) - C (\tau + \delta + \tau k(\delta)) \\
                &\geq \frac{2}{\Lambda}(1-e^{-\Lambda t}) - 3 \eta \\
                &\geq \frac{1}{1 + \epsilon} \frac{2}{\Lambda}(1-e^{-\Lambda t})
            \end{align*}
            this shows that
            \[
                E_t^\tau = \frac{1}{\tau} E_k^\tau \leq (1+\epsilon) \frac{\Lambda}{2(1-e^{-\Lambda t})}
            \]
            for all $\tau < \tau(\epsilon,t_0)$ and $t \geq t_0$, concluding the proof.
        \end{itemize}
    \end{proof}

\end{document}